\documentclass[11pt]{article}

\usepackage{latexsym}
\usepackage{amsthm}
\usepackage{amsmath}
\usepackage{amssymb}
\usepackage{amsfonts}
\usepackage{graphicx}
\usepackage{graphics}
\usepackage{verbatim}
\usepackage{bm}
\usepackage{natbib}

\usepackage{latexsym}
\usepackage{amsmath}
\usepackage{amsfonts}
\usepackage{graphicx}
\usepackage{graphics}
\usepackage{epsfig}
\usepackage{subfigure}
\usepackage{verbatim}
\usepackage{amssymb}
\usepackage{xr}

\newtheorem{theorem}{Theorem}

\def\and{\mbox{and}}

\numberwithin{equation}{section}
\renewcommand{\thefootnote}{\fnsymbol{footnote}}
\renewcommand{\thefootnote}{\fnsymbol{footnote}}
\numberwithin{theorem}{section}
\numberwithin{lemma}{section}
\numberwithin{proposition}{section}
\numberwithin{corollary}{section}
\numberwithin{remark}{section}

\def\and{\mbox{ and }}

\def\ba{\begin{array}}
\def\bc{\begin{center}}
\def\bd{\begin{description}}
\def\be{\begin{enumerate}}
\def\ec{\end{center}}
\def\ed{\end{description}}
\def\ee{\end{enumerate}}
\def\ea{\end{array}}
\def\ben{\begin{equation}}
\def\benn{\begin{equation*}}
\def\een{\end{equation}}
\def\eenn{\end{equation*}}
\def\benr{\begin{eqnarray}}
\def\eenr{\end{eqnarray}}
\def\benrr{\begin{eqnarray*}}
\def\eenrr{\end{eqnarray*}}

\def\edt{\end{document}}

\def\vs{\vskip}

\def\vec#1{\oalign{#1\crcr\hidewidth \vbox to .3ex{\kern 1pt
           \hbox spread .1em
           {\hss\char'176\hss}\vss}\hidewidth}}

\usepackage{mathtools}

\begin{document}



\baselineskip 21pt

\begin{center}
{\Large \bf Consistency of full-sample bootstrap for estimating  high-quantile, tail probability, and tail index}\\[1cm]
{\large Svetlana Litvinova and Mervyn J. Silvapulle} \\
{\em Department of Econometrics and Business Statistics, Monash University, Australia}\\
   {svetlana.litvinova@monash.edu, mervyn.silvapulle@monash.edu}\\[0.3cm]
\end{center}

{
\renewcommand{\thefootnote}{}
\setcounter{footnote}{1}
\footnotetext{Corresponding author:
Svetlana Litvinova, Department of Econometrics and Business Statistics, Monash University, Australia.
   email: svetlana.litvinova@monash.edu.}
\setcounter{footnote}{0}
}

\renewcommand{\thefootnote}{\arabic{footnote}}

\vs .25in



\smallskip

\begin{abstract}
We show that the full-sample bootstrap is asymptotically valid for constructing confidence intervals for high-quantiles, tail probabilities, and other tail parameters of a univariate distribution.
This resolves the doubts that have been raised about the validity of such bootstrap methods.
In our extensive simulation study, the overall performance of
 the bootstrap method was  better than that of the standard asymptotic method,
 indicating that the bootstrap method is at least as good, if not better than, the asymptotic method for inference.
This paper also lays the foundation for developing bootstrap methods for inference about tail events in multivariate statistics; this is particularly important because some of the non-bootstrap methods are complex.
\end{abstract}

\vs .25in


\smallskip


\newpage

\section{Introduction}
\label{intro}
In risk management, the quantities of interest are often related to
 rare and costly extreme events.
Some examples of statistical inference questions that arise in such context are:
(1) What is the probability that an insurance claim would exceed a given large threshold?
(2) What is the level of financial loss that an investment portfolio
    exceeds with only a very small probability, say, 0.001?
The main statistical method used for answering such questions is confidence
 intervals for tail parameters.
For example, it may be desired to construct a confidence interval for the tail probability $p(x) = 1- F(x)$,
  where $F$ is the unknown population distribution function and
  $x$ is  a given large value;
 here large $x$ means that it may be larger than the
  largest order statistic in the sample and hence the sample proportion of observations that exceeds $x$, which is zero,
  is not a useful estimator of $p(x)$.
Doubts have been raised about the validity of full-sample bootstrap for constructing confidence intervals for such tail parameters.
The objective of this paper is to show that the full-sample bootstrap is valid.

There is a large literature on estimation of confidence intervals
 for high-quantiles and tail probabilities.
For  excellent accounts of the topic, see Chapter 4 in \cite{Beir:stat:2004},
 section 4.3 in \cite{Ha:Ferr:2006}, and
 sections 6.3 and 6.4 in \cite{Embr:Klup:Miko:mode:1997}.
 \cite{Cole:an:2001} provides an excellent introductory account to the topic.
A topic that frequently arises in this area is construction of confidence intervals for tail parameters
(for eg., \citealt{Tajvidi-2003}).
Although $m$ out of $n$ and subsample bootstrap are  in general asymptotically valid in
  these cases
 they do not appear to perform well in finite samples (for eg., \citealt{Davidson-2015}).
\cite{Gelu:deHa:2002} examined the validity of the bootstrap method for intermediate order statistics;
see also \cite{Gomes-2016}.
The validity of bootstrap for inference in heavy-tailed distributions is a delicate issue (\citealt{Davidson-2}).
The question of whether or not full-sample nonparametric bootstrap in the iid setting is consistent for
 inference on high quantiles and tail probabilities in univariate distributions, has been
 an open question for sometime.
This paper provides a rigorous answer to this question.

The main contribution of this paper is that it is shown
  that the full-sample bootstrap methods for constructing confidence intervals for
 the tail parameters,
high quantiles, tail probabilities, and the extreme value tail index are valid.
In our simulation study, the overall performance of the proposed bootstrap method
 was better than that of the asymptotic method, but none of them performed uniformly the best.
This corroborates the validity of the main result of this paper, namely that the  bootstrap is valid.
The results in this paper also provide the foundation for developing
 full-sample bootstrap methods for tails of multivariate distributions and for time series.
This is particularly important since methods based on the asymptotic distribution of estimators
of tail parameters in the multivariate setting are typically rather complex.

The method of estimation is based on tail empirical process (\citealt{Einm:1990}).
Our technique for establishing bootstrap validity draws
from the probability theory literature on bootstrapping tail empirical processes
(\citealt{Csor:Maso:1989}).
Therefore, in this short paper we state the main results as five theorems, and  relegate the mathematical details to the Supplementary Materials.
The validity of bootstrap for inference on some parameters of a heavy-tailed distribution has been established
using different machinery, namely wild bootstrap and permutation bootstrap
(see, \citealt{Giuseppe-1,Giuseppe-2}); it would be interest to know whether or not their machinery
could be applied to the setting of this paper.
\cite{Lahiri-2003} provides a chapter on bootstrapping heavy-tailed data and extremes
for time-series and stationary processes.

The paper is structured as follows.
Section \ref{s: Preliminaries} introduces some notation, assumptions, and preliminaries.
Section \ref{sec-3} provides a statement of the main results.
The proofs of these main results are provided in the Supplementary Materials.
Numerical studies are provided in the next two Sections.

\section{Notation, assumptions, and preliminaries}
\label{s: Preliminaries}
A discussion of the results in this section may be found in
\cite{Beir:stat:2004}, \cite{Ha:Ferr:2006},
  \cite{Embr:Klup:Miko:mode:1997}, and
 \cite{Cole:an:2001}, among others.
Let $X_1, X_2, \ldots, X_n $ be independent and identically distributed
 with common distribution function $F$, and let
 $X_{1,,n}\leq X_{2,n}\leq \ldots \leq X_{n,n}$ denote the corresponding
 order  statistics.
If there exists a sequence of constants $a_n>0$ and $b_n \in \mathbb{R}$ such that
$P((X_{n,n}-b_n)/a_n \leq x)  \longrightarrow G(x)$ as $n \rightarrow \infty$,
and $G$ is nondegenerate, then (a) $G$ is called an \textit{extreme value distribution} function,
 (b) $F$ is said to be in the \textit{maximum domain of attraction} of $G$, denoted by $F \in D(G)$, and
 (c) the class of extreme value distributions is $G_{\gamma}(ax+b)$, where
 $G_{\gamma}(x)= \exp\{-(1+\gamma x)^{-1/ \gamma} \}$  ($x > -\gamma^{-1}, b \in \mathbb{R}, a>0$).
The parameter $\gamma$, called the \textit{extreme value index}
 or the \textit{tail index}, characterizes the tail behaviour of $F$;
in what follows, the term 'tail'  refers to 'right tail'.
If $\gamma<0$, $\gamma=0$, or $\gamma>0$  then $F$ has a finite end point, \textit{light tail}, or \textit{heavy tail}, respectively.
\textit{In what follows, we assume that $F \in D(G_{\gamma})$ for some $\gamma \in \mathbb{R}.$}

Let $U(y) =  F^{\leftarrow}(1 - y^{-1})$, where $F^{\leftarrow}$ denotes the inverse function of $F$.
There is a one-to-one correspondence between $F$ and $U$.
Often, regularity conditions on the tails of $F$ are expressed in terms of $U$.
The distribution function $F \in D(G_{\gamma})$ if and only if there exists
  a positive function $a(t)$ ($t \in \mathbb{R}$) such that (for example, Theorem 1.1.6, \citealt{Ha:Ferr:2006}).
\begin{equation}
\lim_{t\rightarrow \infty} \{U(tx)-U(t)\}/a(t)=\{x^{\gamma}-1\}/\gamma, \quad (x>0).
\label{eq:r_mda2}
\end{equation}
If (\ref{eq:r_mda2}) holds, then $U$ is said to satisfy the \textit{first-order condition} of regular variation.
In what follows, we assume that this is satisfied.

The asymptotic  methods of statistical inference on  the tail parameters
 are based on the asymptotic normality of estimators of such parameters.
In order to establish the asymptotic normality of such estimators,
 a second-order refinement of (\ref{eq:r_mda2}) is also assumed.
The function $U$ is said to satisfy the \textit{second-order condition of regular variation} if there exists a positive or negative function $A(\cdot)$ with  $\lim_{t \rightarrow \infty} A(t)=0$, such that
\begin{equation}
\lim_{t\rightarrow \infty} [ \{U(tx)-U(t)\}/a(t) - \{ x^{\gamma}-1 \} / \gamma]/A(t)= H_{\gamma, \rho}(x) \quad (x>0),
\label{eq:r_soc2-main}
\end{equation}
where $H_{\gamma, \rho}(x)=\int_1^x s^{\gamma-1}\int_1^s u^{\rho-1} du ds$, $\rho \leq 0$ is called the second order parameter of regular variation, and $A(t)$ is regularly varying with index $\rho$.
Let $\gamma_{-} = \min\{0, \gamma\}$ and $\gamma_{+} = \max\{0, \gamma\}.$
\textit{In what follows we assume that (\ref{eq:r_soc2-main}) is satisfied, $\gamma \neq \rho$, and that $\rho < 0$
 if $\gamma > 0.$}
Under the additional assumption (\ref{eq:r_soc2-main}), we have
 [page 103 and Lemma B.3.16 in \citealt{Ha:Ferr:2006} ]
\begin{equation}
\lim\limits_{t\rightarrow \infty}  \big[ {\{\log U(tx)- \log U(t)\}/{q(t)}- ({x^{\gamma_-}-1})/{\gamma_-}} \big]/Q(t)=  H_{\gamma_-, \rho'}(x) \quad \ (x>0),
\label{eq:r_soc33}
\end{equation}
and $Q(\cdot)$ does not change sign eventually with $Q(t) \rightarrow 0$ as $n \rightarrow \infty$.
When $\gamma>0$ and $\rho= 0$ the limit (\ref{eq:r_soc33}) vanishes.

Let $k(n)$ denote a sequence of positive integers satisfying $k(n) \rightarrow \infty$ and $\{k(n)/n\}\rightarrow 0$ as $n \rightarrow \infty$; in what follows we write $k$ for $k(n).$
Then $X_{n-k, n}$ is called an \textit{intermediate order statistic}, and
the sequence $\{k\}$ is called an \textit{intermediate order sequence}.
The estimators of tail quantities presented in this paper are all
 based on the largest $(k+1)$  order statistics $\{X_{n-k,n}, \ldots, X_{n,n}\}$.
For asymptotic normality to hold, $k$  must satisfy some conditions.
These are stated below in which (A.4) implies (A.3).

\noindent \textbf{Condition A:} (A.1). $k \rightarrow \infty$ and $(k/n)\rightarrow 0$ as $n \rightarrow \infty$.
 (A.2). The function $U$ satisfies  (\ref{eq:r_mda2}) and (\ref{eq:r_soc2-main}).
 (A.3). $k^{1/2} A(n/k) \rightarrow 0$ as $n \rightarrow \infty.$
 (A.4). $k^{1/2} Q(n/k) \rightarrow 0$ as $n \rightarrow \infty.$
(A.5). If $\gamma < -1/2$ then $k^{ |\gamma| }Q_0(n/k) = O(1).$

Next we introduce a condition on the distribution function $F$, and assume that this holds in what follows.

\noindent \textbf{Condition B}.  The function $U$ corresponding to $F$ satisfies the first and second  order conditions of regular variation, (\ref{eq:r_mda2}) and (\ref{eq:r_soc2-main}) respectively.
Further, $F$ and its probability density function $f$ satisfy the following conditions
( these are the same as those in Lemma 6.1.1 of \citealt{Cs-Horvath-1993} on page 369, and also in Proposition 2.4.9 of \citealt{Ha:Ferr:2006}):
(i) $F$ is differentiable on $(x_*, x^*)$ where $x_* = \sup\{x: F(x) = 0\}$, $x^* = \inf\{x: F(x) = 1\}$,
and $-\infty \leq x_* < x^* \leq \infty.$ (ii) $f(x) > 0$, $(a<x<b)$.
(iii) $\sup_{0 < t < 1} t(1-t){f'\{Q(t)\}}/{f^2\{Q(t)\} }< C $ for some $C >0$, where $Q$ is the quantile function of $F$ defined by
$Q(t) = F^{\leftarrow}(t) = \inf\{x: F(x) \geq t\},  (0<t<1), Q(0) = Q(0+).$

The estimators of tail parameters studied in this paper are based on the method of moments, which is a
 widely used general method of inference in this area (\citealt{Einm:Kraj:Sege:2008}, \citealt{Dekk:Einm:deHa:1989}).
Derivations of these estimators may be found in section 3.5 of \cite{Ha:Ferr:2006};
see (3.2.2), (3.5.2), and (3.5.9) in \cite{Ha:Ferr:2006}.
Let us introduce the following statistics and estimators:
$ H_{n} = k^{-1} \sum_{i=0}^{k-1}[\log X_{n-i,n}-\log X_{n-k,n}], \ $
$M_{n} = k^{-1} \sum_{i=0}^{k-1}\left(\log X_{n-i,n}-\log X_{n-k,n}\right)^2, \ $
$\hat{\gamma}_+ =  H_{n}, \ \hat{\gamma}_- =  1 -2^{-1}\{  1-({H_{n}^2}/{M_{n})} \}^{-1},
\quad \hat{\gamma} = \hat{\gamma}_+ + \hat{\gamma}_-, \ $
$\hat{a}(n/k) = X_{n-k,n}H_{n}(1-\hat{\gamma}_-)$,
and $\hat{b}(n/k) = X_{n-k,n}$.

\section{Asymptotic validity of bootstrap}\label{sec-3}

Let $X_1^*, \ldots, X_n^*$ denote a simple random sample from the empirical distribution function
 $F_n$ of $\{X_1, \ldots, X_n\}$,
and let $X_{1,n}^*, \ldots, X_{1,n}^*$ denote the corresponding order statistics.
Let $\{\hat b^*(n/k),\hat{\gamma}^*,$ $ \hat a^*(n/k), \hat{x}^*_n, \hat{p}^*_n\}$ denote the bootstrap statistics corresponding to $\{\hat b(n/k),\hat{\gamma}, \hat a(n/k), \hat{x}_n, \hat{p}_n\}$.
Let $P^*$ denote the bootstrap probability conditional on the sample $\{X_1, \ldots, X_n\}$.
The following results establish the consistency of the full-sample bootstrap to estimate asymptotic distributions of $\hat b(n/k),\hat{\gamma}, \hat a(n/k), \hat{x}_n$, and $\hat{p}_n$.
In each of these theorems, the convergence in distribution of the non-bootstrap statistic is already known;
 references to the corresponding non-bootstrap result is provided in the
 Supplementary materials at the beginning of the proof of each theorem.

\begin{theorem}
\label{th:mm_b_boot}
Suppose that (A.1), (A.2), (A.3), and (\ref{eq:r_soc2-main}) are satisfied
   for some $\gamma \in \mathbb{R}$, $\rho \leq 0$, and  $\rho \neq \gamma$.
Then,
$P^*[Y_n^* \leq y  \mid X_1, \ldots, X_n] - P[Y_n \leq y]
  \rightarrow 0$, in probability, as $n \rightarrow \infty,$
where $Y_n^* =  k^{1/2}\{\hat b^*(n/k) -\hat b(n/k)\}/\hat a(n/k)$ and
$Y_n= k^{1/2}\{\hat b(n/k)-U(n/k)\}/ \hat a(n/k)$,  ($y \in \mathbb{R}$).
\end{theorem}
Since $\hat b(n/k) = X_{n-k}$, it follows from the above theorem that bootstrap method consistently estimates the
 asymptotic distribution of the intermediate order statistic $X_{n-k}.$

\begin{theorem}
\label{th:mm_gamma_boot}
Suppose that  (A.1), (A.2), (A.4), and
(\ref{eq:r_soc2-main}) are satisfied for some $\gamma \in \mathbb{R}$,
  $\rho \leq 0$, and  $\rho \neq \gamma$.
Then,
 $P^*[Y_n^* \leq y  \mid X_1, \ldots, X_n] - P[Y_n \leq y]
  \rightarrow 0$, in probability,  as $n \rightarrow \infty  \quad (y \in \mathbb{R}),$
where $Y_n^* =  k^{1/2}(\hat{\gamma}^*-\hat \gamma)$ and
$Y_n= k^{1/2}(\hat \gamma - \gamma).$
\end{theorem}
It follows from the foregoing theorem
that asymptotically valid confidence intervals for the tail index $\gamma$ may be constructed by bootstrap.
\begin{theorem}
\label{th:mm_a_boot}
Suppose that  (A.1), (A.2), (A.4),
and  (\ref{eq:r_soc2-main}) are satisfied for some $\gamma \in \mathbb{R}$,
   $\rho \leq 0$, and $\rho \neq \gamma$.
Then,
$P^*[Y_n^* \leq y  \mid X_1, \ldots, X_n] - P[Y_n \leq y]
  \rightarrow 0,$ in probability, as $n \rightarrow \infty  \quad (y \in \mathbb{R}),$
where $Y_n^* =   k^{1/2}\left[\hat{a}^*(n/k)/\hat a(n/k))-1\right]$ and
$Y_n= k^{1/2}\left[\hat{a}(n/k)/ a(n/k))-1\right].$
\end{theorem}

Next, we consider estimation of a high quantile.
Let $p_n$ be a given small number in the range $0 < p_n < 1$;
 exactly what is meant by 'small' is made more precise later.
For now, we may think of $p_n$ being close to or even smaller than $(1/n).$
Let $x(p_n) = F^{\leftarrow}(1-p_n)$ denote the upper $p_n$th quantile of $F$.
We refer to $x(p_n)$ as a high-quantile since $p_n$ is small.
Let $d_n = k/(np_n)$.
An estimator of $x(p_n)$ is
$\hat{x}(p_n)  = \hat{b}\left({n}/{k}\right) + \hat{a}\left({n}/{k}  \right) \{d_n^{\hat\gamma}-1\}/{\hat\gamma}$;
see (4.3.3 in \citealt{Ha:Ferr:2006} for arguments leading to this estimator.
If $\hat{\gamma} = 0$ then $(d_n^{\hat\gamma}-1)/{\hat\gamma}$ is interpreted as $\log d_n.$
Let $ q_{\gamma}(t) = \int_1^t s^{\gamma-1}\log s\ ds,$  $( t>0).$
The asymptotic distribution of  $\hat{x}(p_n)$ is provided in the next theorem;
 it can be used for constructing a confidence interval for the quantile $x_n.$

\begin{theorem}
\label{th:mm_x_boot}
Suppose that (i) (A.1), (A.2), and (A.4) are satisfied,
 (ii) (\ref{eq:r_soc2-main}) is satisfied for some $\gamma \in \mathbb{R}$ and $\rho < 0$ or $\rho=0$ with $\gamma<0$, and
(iii) $np_n=o(k)$ and $\log (np_n) = o(k^{1/2})$.
 Let $x_n = x(p_n), \hat{x}_n = \hat{x}(p_n)$, and $\hat{x}^*_n = \hat{x}^*(p_n)$
Then, as $n \rightarrow \infty$,
$
P^*[Y_n^* \leq y  \mid X_1, \ldots, X_n] - P[Y_n \leq y]  \rightarrow 0$,
in probability,  as  $n \rightarrow \infty  \quad (y \in \mathbb{R}), $
where
$ Y_n^* =    k^{1/2} (\hat{x}^*_{n}-\hat x_{n})/ \{(a(n/k )q_{\hat\gamma}(d_n)\}$
and
$ Y_n= k^{1/2} (\hat{x}_{n}-x_{n}) / \{a(n/k )q_{\hat\gamma}(d_n)\}.$
\end{theorem}

Next, we consider estimation of tail probability.
Let $x_n$ be a given large number, and consider estimation of the tail probability
 $p(x_n) = 1 - F(x_n)$.
 Let
\begin{equation}
\hat{p}(x_n)  = ({k}/{n}) \left(  \max \left\{  0, 1+\hat{\gamma}
{ \{x_{n}-\hat{b}\left({n}/{k}\right)\}}/{  \hat{a}\left({n}/{k}  \right) } \right\} \right)^{-1/\hat{\gamma}};
\label{eq:r_hat_p}
\end{equation}
see section 4.4 in \cite{Ha:Ferr:2006} for some details leading to this estimator.
Let $w_{\gamma}(t) = t^{-\gamma}\int_1^t s^{\gamma-1}\log s \ ds, \quad (t>0)$.
The asymptotic distribution of the estimator $\hat{p}(x_n)$ is given in the next result; it can be used for constructing a confidence interval for $p(x_n).$

\begin{theorem}
\label{th:mm_p_boot}
Suppose that
(i) (A.1), (A.2), and (A.4) are satisfied,
 (ii) (\ref{eq:r_soc2-main}) is satisfied for some $\gamma  > -(1/2)$,
(iii)  $\rho < 0$, or $\rho=0$ with $\gamma<0$,
and (iv) $d_n \rightarrow \infty$ and $w_{\gamma}(d_n) =o(k^{1/2})$.
Let $p_n = p(x_n), \hat{p}_n = \hat{p}(x_n)$, $\hat{p}^*_n = \hat{p}^*(x_n),$ and $\hat{d}_n = k/(n\hat{p}_n).$
Then,
$P^*[Y_n^* \leq y  \mid X_1, \ldots, X_n] - P[Y_n \leq y]
  \rightarrow 0$,  in probability, as $n \rightarrow \infty  \quad (y \in \mathbb{R}),$
where
$Y_n^* =  k^{1/2} \left[(\hat{p}_n^*/\hat p_n)-1 \right]/ w_{\hat\gamma^*}(\hat{d}_n^*)$
and
$Y_n= k^{1/2} \left[(\hat{p}_n/ p_n)-1 \right]/ w_{\hat\gamma}(\hat{d}_n).$
\end{theorem}

Based on Theorems  \ref{th:mm_b_boot} -- \ref{th:mm_p_boot}, asymptotically valid  confidence interval for $ b(n/k),\gamma, a(n/k), {x}(p_n)$ for a given small $p_n$, and ${p}(x_n)$ for a given large $x_n$, can be constructed by full-sample bootstrap.
In this paper we study performance of the percentile, basic, and t-bootstrap methods for ${p}_n$ by a simulation study.

To construct a confidence interval for the tail probability $p(x_n)$, where $x_n$ is large number,
 suppose that the conditions of
 Theorem \ref{th:mm_p_boot} are satisfied.
Let $\hat{p}^*_{n,\alpha}$ denote the $\alpha$-quantile of the distribution
 $P^*\{\hat{p}^*_n \leq x |  X_1, \ldots, X_n\}.$
Then, the \textit{Efron's percentile} $(1-\alpha)$-confidence interval for $p(x_n)$ is
$I_{Ep}(\alpha) = \left(\hat{p}^*_{n,\alpha/2}, \hat{p}^*_{n,1-\alpha/2}\right)$.
Let $\zeta_{n,\alpha}$ denote the $\alpha$-quantile of the distribution
$P^*\big( k^{1/2} \log (\hat{p}^*_{n}/ \hat{p}_{n} ) / w_{\hat\gamma}(d_n) \leq x  |  X_1, \ldots, X_n \big),$
then the \textit{percentile} $(1-\alpha)$-confidence interval for $p_n$ is
\begin{equation}
I_{p}(\alpha) = \left( \hat{p}_{n} \exp\left\{-\zeta_{n,1-\alpha/2}w_{\hat\gamma}(d_n) k^{-1/2} \right \}, \hat{p}_{n}\exp\left\{-\zeta_{n,\alpha/2} w_{\hat\gamma}(d_n) k^{-1/2} \right \}  \right).
\label{eq: p_ci}
\end{equation}

 To define the t-bootstrap confidence interval, first note that
the limiting distribution in Theorem \ref{th:mm_p_boot} is in fact normal with mean zero and variance
$$\sigma^2(\gamma) =
( \gamma^2+1)I(\gamma \geq 0) +
     \frac{(1-\gamma)^2(1-3\gamma+4\gamma^2)}{(1-2\gamma)(1-3\gamma)(1-4\gamma)}I(\gamma<0)
$$
where $I(\cdot)$ denotes the indicator function
( page 141, \citealt{Ha:Ferr:2006}).
Let $\hat{\sigma}^2 = \sigma^2(\hat{\gamma})$, and
let $\xi_{n,\alpha}$ denote the $\alpha$-quantile of the distribution
$P^*\left(k^{1/2} \log (\hat{p}_n^*/\hat p_n) /( w_{\hat\gamma}(d_n) \hat\sigma) \leq x  |  X_1, \ldots, X_n  \right).$
Then a $(1-\alpha)$-level \textit{t}-confidence interval for $p_n$ is
\begin{equation}
I_{t}(\alpha) = \left( \hat{p}_{n}\exp\left\{-\xi_{n,1-\alpha/2} w_{\hat\gamma}(d_n) \hat\sigma k^{-1/2} \right \}, \hat{p}_{n}\exp\left\{-\xi_{n,\alpha/2}w_{\hat\gamma}(d_n) \hat\sigma k^{-1/2} \right \}  \right).
\label{eq: t_ci}
\end{equation}

 \section{Application to  Danish Fire Insurance Data}

To illustrate the bootstrap method for constructing confidence intervals for a tail probability we use Danish fire insurance data.
The data consists of insurance claims  exceeding one million Danish Krone (DKK);
each claim corresponds to total losses due to damages to buildings and contents,  and loss of profits.
There are $n=2156$ observations in the data set.
The data set is well known and studied, see \cite{McNe:1997},  \cite{Resn:1997},
and \cite{Lee:Qi:2019}, for example.


%
%
\begin{figure}
\caption{\small The 95\% confidence interval for $p_n=P(X>300)$.
The confidence intervals obtained by the percentile bootstrap and by the normal approximation methods are given as functions of $k$ ($k=20,25, \ldots, 300$).
The dotted line is the estimate of the tail probability $p_n$.
The solid lines are upper and lower bounds of the percentile bootstrap confidence interval.
The dashed lines are upper and lower bounds of the asymptotic confidence intervals.}
\begin{center}
\includegraphics[scale=0.55]{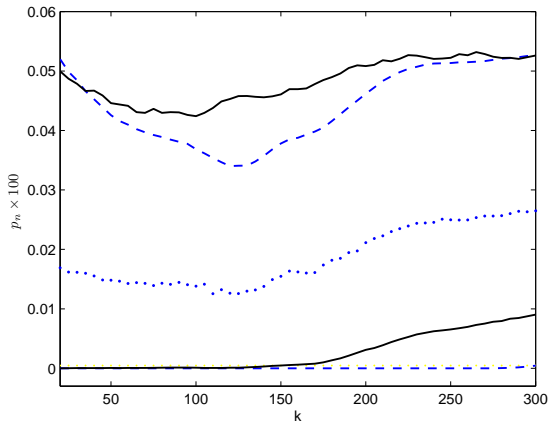}
\label{F:application-2}
\end{center}
\end{figure}

%

%
%

For illustrative purposes, suppose that we wish to construct a 95\% confidence interval for the probability of a loss over 300 million DKK.
There are no observations exceeding this level and hence a method based on extreme value theory is essential; we have also performed the calculations for other thresholds.
Figure \ref{F:application-2} provides pointwise confidence intervals obtained by the percentile bootstrap and asymptotic methods.
The asymptotic confidence interval is based on $\{k^{1/2}/w_{\hat{\gamma}}(\hat{d}_n)\} (p_n \hat{p}_n^{-1} - 1)$  being asymptotically  normal.
The tail probability estimates for  the asymptotic method are stable for $k$ ranging from 45 to 105.
Confidence intervals obtained by the bootstrap and the asymptotic methods are close.
The bootstrap confidence intervals are slightly more stable across different values of $k$.
The asymptotic confidence intervals are slightly shorter on average; but this does not mean that they are better since one needs to take coverage probability into account.
Since, in the simulation study, the bootstrap method provided better coverage than the asymptotic method,
 the indications are that the former is more reliable than the latter.
We also carried out the calculations for $pr(X>100)$.
Out of the 2156 observations, there are only 3 that exceed this level;
therefore, again methods based on extreme value theory are essential.
For this case also, the relative performance of the two methods is essentially the same.

\begin{figure}
\caption{Coverage rates of the asymptotic and bootstrap 95\% confidence intervals for the tail probability $1-F(x_n)$, where $x_n$ was chosen such that $p_n:=1-F(x_n) = 0.0005$ and $n=1000.$}
\begin{center}
\begin{footnotesize}
\begin{tabular}{ll}
 \epsfig{file=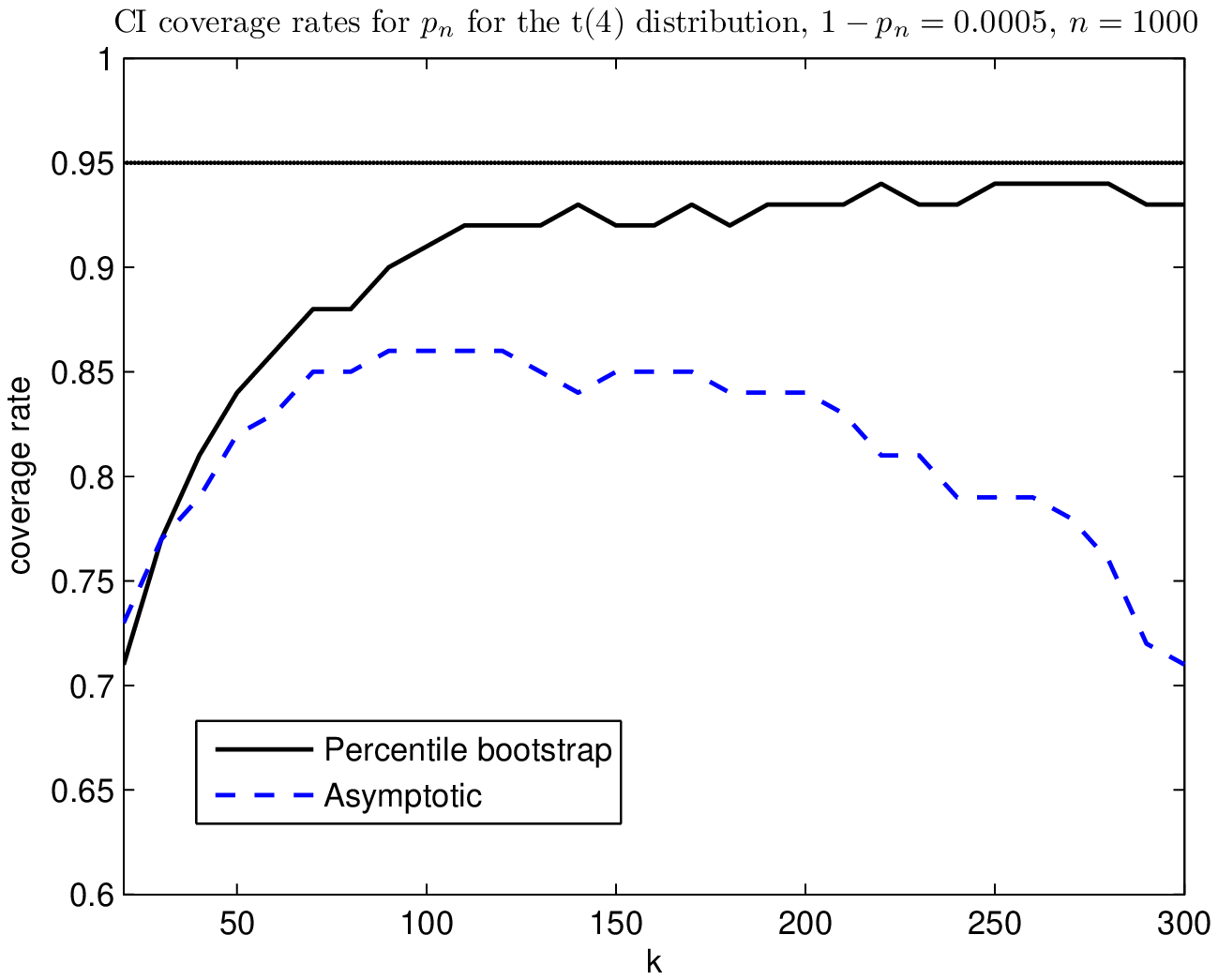,width=0.45\linewidth,clip=} &
 \epsfig{file=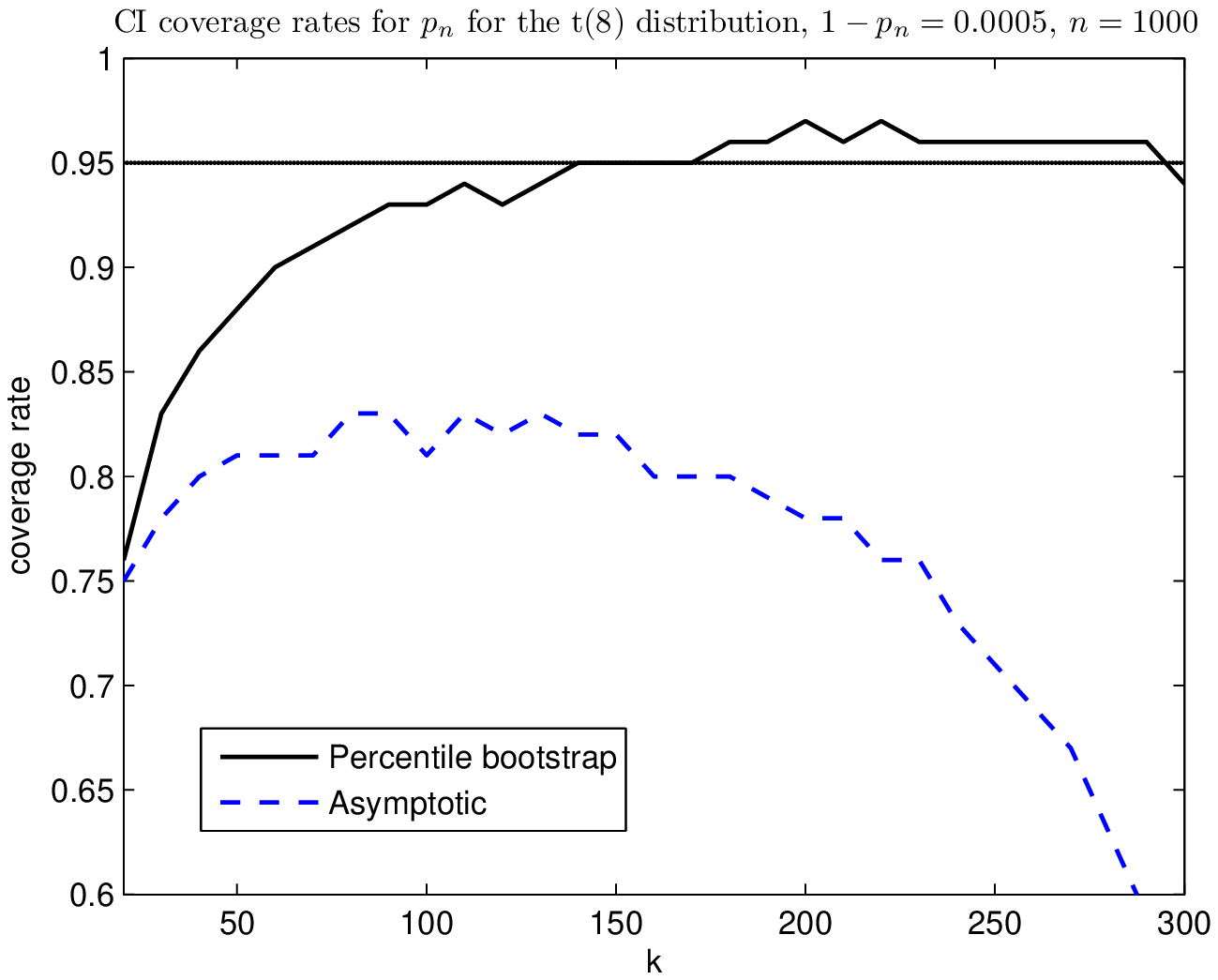,width=0.45\linewidth,clip=} \\
b) \epsfig{file=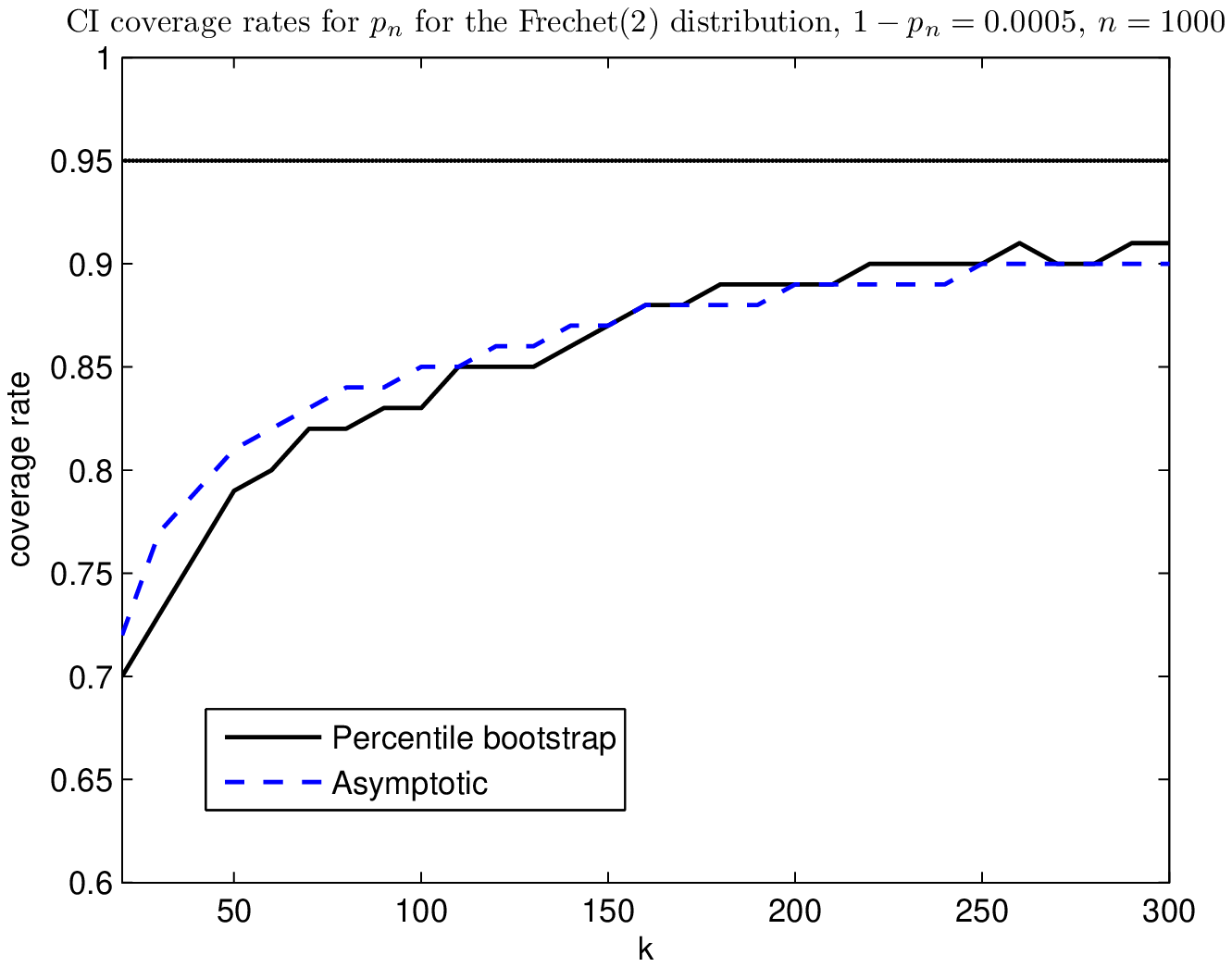,width=0.45\linewidth,clip=}  &
\epsfig{file=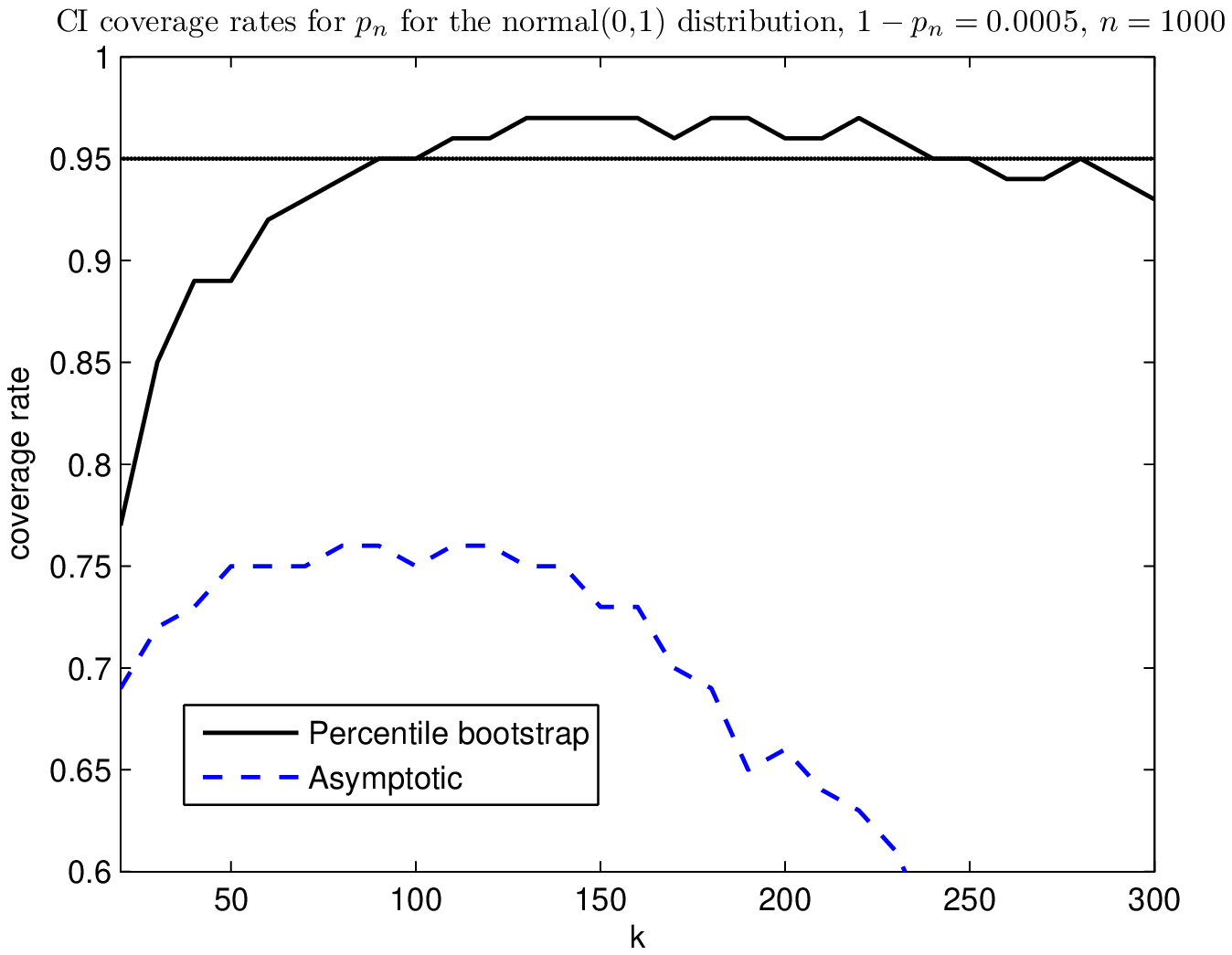,width=0.45\linewidth,clip=}\\
c)\epsfig{file=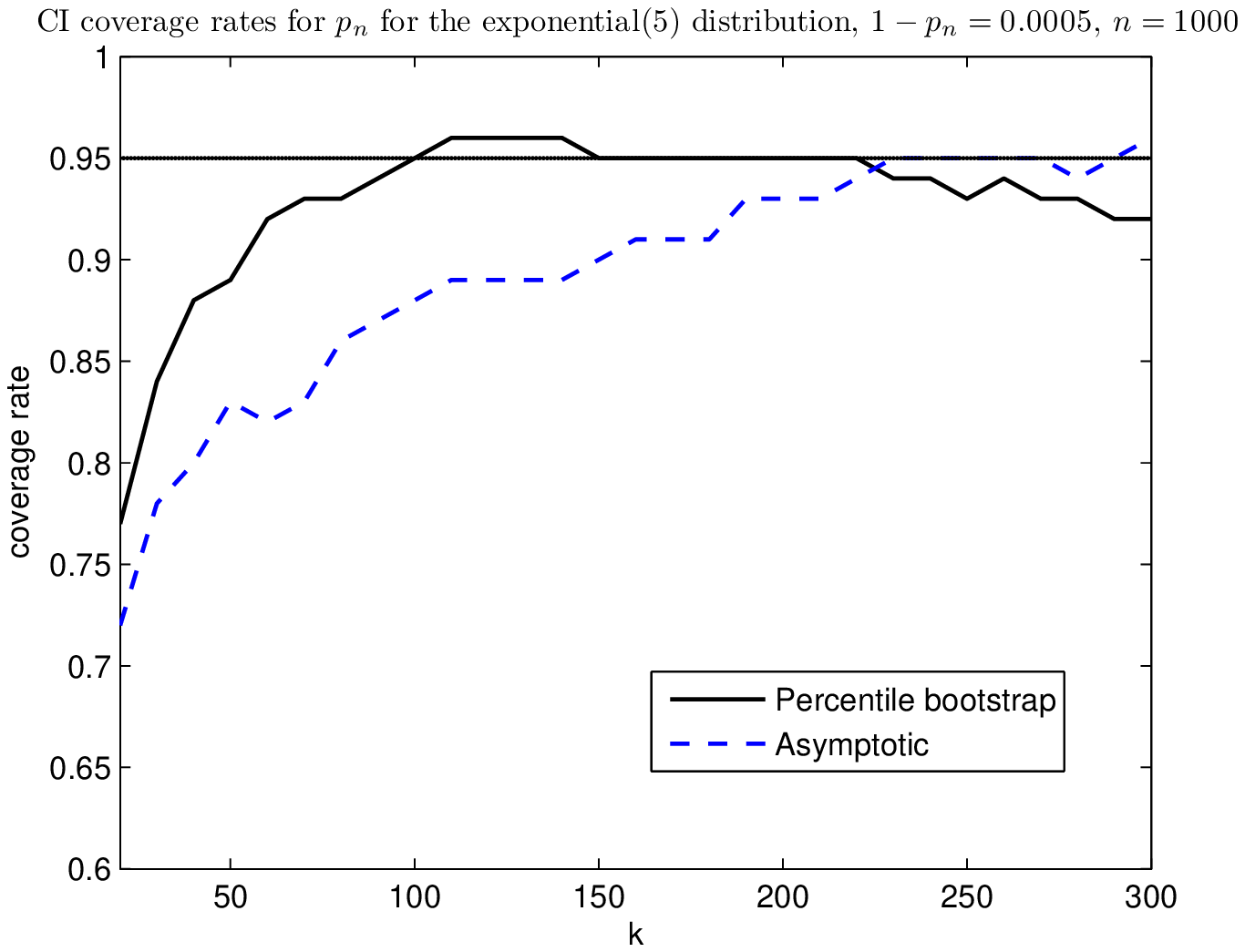,width=0.45\linewidth,clip=}   &
\epsfig{file=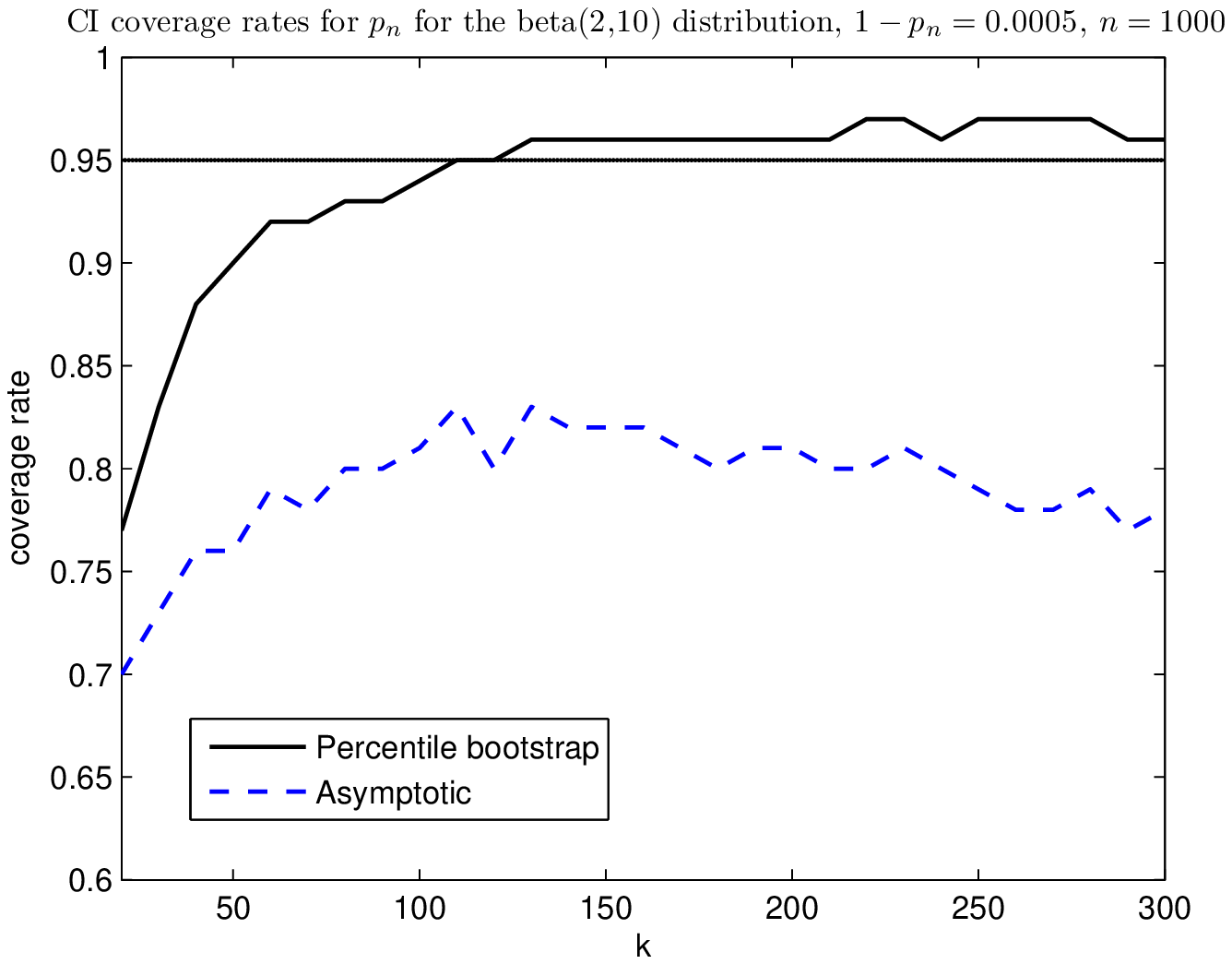,width=0.45\linewidth,clip=}
\end{tabular}
\end{footnotesize}
\end{center}
\label{Fig-1}
\end{figure}

\section{Simulation Study}
We carried out a simulation study to evaluate the reliability of the bootstrap method.
First, we outline the design of the study and then present a summary of the results;
detailed results of the study, mostly in the form of graphs, are presented in the Supplementary Materials.

\noindent \textit{Design}

For the population distribution function $F$, we considered the following cases.
(I) Heavy tail ($\gamma >0$): $t(2), t(4), t(8),$ Frechet(1), and Frechet(2).
[For $t(\alpha$) and Frechet($\alpha$), the tail index $\gamma = \alpha^{-1}.$]
(II) Light tail ($\gamma =0 $): exp(5) and $N(0,1)$.
(III) Bounded support ($\gamma < 0$): beta(2, 10), beta(1, 2).

For each distribution, we considered sample sizes $n=200, 1000$.
First, we choose a small tail probability, denoted $p_n$, and compute $x_n = F^{-1}(1 - p_n)$.
The unknown tail quantity to be estimated is $p_n$ that corresponds to the known high-quantile, $x_n$.
We considered $np_n = 1, 0.5, 0.25$ when $n = 1000$, and $np_n = 5, 1, 0.5$ when $n = 200.$
To this end, we apply Theorem \ref{th:mm_p_boot}.
The main simulation steps are:
 (1) Choose a population distribution function $F$, sample size $n$, and a small tail probability, $p_n$.
 (2) Generate a simple random sample $X_1, \ldots, X_n$ from $F$.  For each given $k$ in a range of values, estimate $p_n$ corresponding to $x_n$, and a confidence interval based on the asymptotic normality of the estimator.
 (3) Draw a simple random sample $\{X_1^*, \ldots, X_n^*\}$ from $\{X_1, \ldots, X_n\}$.
 (4) For each $k$, compute the estimate $\hat{p}_n^*$.
 (5) Repeat the previous two steps 1000 times and compute the percentile, basic, and $t$ bootstrap confidence intervals, for each $k$.

\bigskip

\noindent \textit{Results}

\smallskip

The coverage rates for $n=1000$ and $1-p=0.0005$ are presented in Figure \ref{Fig-1}; the results for $1-p=0.0001, 0.00025$ and $n=200$ are presented in a
working paper of the authors.
Here, we present a summary of the main observations.
We use the Efron's percentile bootstrap method for comparison;
the other ones mentioned in Section \ref{sec-3} did not perform as well.
The two main observations of the simulation study:
\newline (1) For every model considered, the bootstrap method performed reasonably well in terms of coverage rate.  More specifically, the observed coverage rate was reasonably close to the nominal level compared to other simulation results published in this area.
\newline (2) The bootstrap method performed better than the asymptotic method in most cases, and in the remaining small number of cases the difference was small.
More specifically, the bootstrap method had better coverage rates than the asymptotic method for all the distributions,
except for some cases of Frechet(1) and Frechet(2) for which the two coverage rates were close.
In  these cases,
 the bootstrap method was slightly better than the asymptotic method for smaller tail probabilities,
  and  the asymptotic method was slightly better for less extreme tail probabilities.
In general, the lighter the tail of the population distribution, the more advantage the bootstrap method had over the asymptotic method in terms of coverage rate.
In terms of length of the confidence interval, the two methods were comparable.
Overall, the bootstrap method performed better than the asymptotic method.

\section*{Acknowledgement}
We wish to thank Professor John Einmahl for helpful comments and suggestions.
Svetlana Litvinova would like to thank Monash University for support through a postgradauate scholarship; this paper is based on her doctoral dissertation.
Mervyn Silvapulle gratefully  acknowledges support from the Australian Government through the Australian Research Council.
We acknowledge that, in 2018, Professor Chen Zhou of Erasmus University (joint work with Laurens de Haan and Alex Koning) presented a seminar at Monash University which contained some results that overlap with those presented in this paper.

\bibliographystyle{natbib}
\bibliography{Regression}

\end{document}




\baselineskip 21pt

\begin{center}
{\Large \bf Supplementary materials for "Consistency of full-sample bootstrap tail estimators:  high-quantile, tail probability, and tail index"}\\[1cm]
{\large Svetlana Litvinova and Mervyn J. Silvapulle} \\
{\em Department of Econometrics and Business Statistics, Monash University, Australia}\\
   {svetlana.litvinova@monash.edu, mervyn.silvapulle@monash.edu}\\[0.3cm]
\end{center}

{
\renewcommand{\thefootnote}{}
\setcounter{footnote}{1}
\footnotetext{Corresponding author:
Svetlana Litvinova, Department of Econometrics and Business Statistics, Monash University, Australia.
   email: svetlana.litvinova@monash.edu.}
\setcounter{footnote}{0}
}

\renewcommand{\thefootnote}{\arabic{footnote}}

\vs .25in


\baselineskip 22pt

\smallskip

\begin{abstract}
The proofs of the results in the main paper are provided in this Supplement.
Initially, some known results are stated.
Then, several lemmas are established, and then the proofs for the main results are provided.
\end{abstract}

\noindent \textit{Keywords}: Full-sample bootstrap; Intermediate order statistic; Extreme value index; Hill estimator; Tail probability; Tail quantile .



%
%
%
%
%
%
%
%
%
%
%
%

\section{Some known results on tail quantile processes.}
\label{sec:App-1}

Let $V_1, \ldots, V_n$ be independent and identically distributed random variables with the Uniform(0,1) distribution.
Let $V_{1,n}\leq \ldots, V_{n,n}$  denote the order statistics of $V_1, \ldots, V_n$.
Let $[x]$ denote the smallest integer larger than or equal to $x$.
Let  $V_n(s) = V_{[ns],n} \quad (0 < s \leq 1)$ with $V_n(0) =V_{1,n}$ denote the
\textit{empirical quantile function}.
Define the \textit{empirical quantile process} of $\{V_1, \ldots, V_n\}$ by
$\beta_n(s)=n^{1/2}\{s-V_n(s)\} \quad (0 \leq s \leq 1)$.
A result on $\beta_n$ that we use is the following:
[\citealt{Csor:Csor:Horv:Maso:1986}, Theorem 2.1]
There exists a probability space $(\Omega, \mathcal{A}, P)$ which carries a sequence of independent uniform (0,1) random variables $\{V_i\}_{i \in \mathbb{N}}$ and a sequence of Brownian bridges
 $\{B_i\}_{i \in \mathbb{N}}$, where $\mathbb{N}$ denotes the set of positive integers, such that
\begin{equation}
\sup_{\lambda/n \leq s \leq 1-\lambda/n}  {  \left|  \beta_n(s)-B_n(s) \right|}/{\{s(1-s)\}^{1/2-\nu}} =O_P(n^{-\nu}), \ (0<\lambda<\infty; 0 \leq \nu <1/2).
\label{eq:qp}
\end{equation}
%
\cite{Csor:Maso:1989} established the following result similar to (\ref{eq:qp}) for the bootstrap version of the uniform empirical quantile process:
There exists a probability space, denoted
  $(\Omega', \mathcal{A}', P')$  that carries the aforementioned sequences $\{V_i\}_{i \in \mathbb{N}}$ and $\{B_i\}_{i \in \mathbb{N}}$, together with a sequence of independent uniform (0,1) random variables
  $\{V_i'\}_{i \in \mathbb{N}}$ and a sequence of Brownian bridges $\{B_i'\}_{i \in \mathbb{N}}$ such that the two sets of random elements
   $\{V_i\}_{i \in \mathbb{N}} \cup \{B_i(s): \ 0 \leq s \leq 1 \}_{i \in \mathbb{N}} $
   and $\{V_i'\}_{i \in \mathbb{N}}  \cup \{B_i'(s): \ 0 \leq s \leq 1 \}_{i \in \mathbb{N}} $
   are independent.
Let $\{V_{i,n}', \beta_n'(s), V_n'(s)\}$  be defined in the same way as
$\{V_{i,n}, \beta_n(s), V_n(s)\}$ except that $\{V_i\}_{i \in \mathbb{N}} $ is replaced by $\{V_i'\}_{i \in \mathbb{N}}$ ($i=1,\ldots,n; 0\leq s \leq 1$).

Let $V_i^*=V_n(V_i')$ and $V_n^*(s)=V_n\{V_n'(s)\}  \ (i=1,\ldots,n; \ 0 \leq s \leq 1).$
Then, $\{V_1^*, \ldots, V_n^*\}$ is a simple random sample from $\{V_1, \ldots, V_n\};$
in what follows, we treat $\{V_1^*, \ldots, V_n^*\}$ as a bootstrap sample from $\{V_1, \ldots, V_n\}$,
and $V_n^*(s)$  as the corresponding \textit{bootstrap empirical quantile function}.
Let $\{V_{1,n}^*, \ldots V_{n,n}^*\}$ denote the order statistics of $\{V_1^*, \ldots V_n^*\}$.
Then, we have $V_{[ks],n}^*=V_n(V_{[ks],n}')=V_n(V_{[nks/n],n}')=V_n\{V_n'(ks/n)\}= V_n^*(ks/n).$

Define the \textit{bootstrap uniform quantile process} as
$\beta_n^*(s)=n^{1/2}\{V_n(s)-V_n^*(s)\}$ ($0 \leq s \leq 1$).
\cite{Csor:Maso:1989} (see Theorem 2.1) showed that
\begin{equation}
\sup_{\lambda/n \leq s \leq 1-\lambda/n}  {  \left|  \beta_n^*(s)-B_n'(s) \right|}/{\{s(1-s)\}^{1/2-\nu}} =O_{P'}(n^{-\nu}) \quad (0<\lambda<\infty; 0 \leq \nu <1/4)
\label{eq:bqp}
\end{equation}
where the $P'$ in $O_{P'}(n^{-\nu}) $ refers to the probability in the space that
carries both $\{V_i, B_i(\cdot)\}_{i \in \mathbb{N}}$ and $\{V_i', B_i'(\cdot)\}_{i \in \mathbb{N}}$.

Let $k$ denote an \textit{intermediate order sequence}; therefore
$k \rightarrow \infty$ and $ k/n\rightarrow 0 \ \mbox{as} \ n\rightarrow \infty.$
Define the \textit{uniform tail quantile process} as
$\beta_{k,n}(s) = ({n}/{k})^{1/2} \beta_n ({ks}/{n})$
and  the \textit{uniform tail quantile  bootstrap processes} as
$\beta_{k,n}^*(s) = (n/k)^{1/2} \beta_n^* (ks/n) $ $ \ (0 < s \leq 1)$.
Then, by simple substitution, we obtain
$\beta_{k,n}(s)  = k^{1/2}\{ s - (n/k)V_{[ks],n}\}$ and
$\beta_{k,n}^*(s) = k^{1/2}\{ (n/k)V_{[ks],n} - (n/k)V_{[ks],n}^*\}$  ($0 \leq s \leq 1$).

Let $B_{k,n}(s)=(n/k)^{1/2}B_n(ks/n)$ and $B_{k,n}'(s)=(n/k)^{1/2}B_n'(ks/n)$ ($0 < s \leq 1$)
where the Brownian bridges $B_n$ and $B_n'$ are the ones introduced earlier to approximate the
quantile processes $\beta_n$ and $\beta_n'$  ($n = 1, 2, \ldots $) respectively.
Then, we have (Theorem 2.3 in  \citealt{Peng:Qi:2017})
\begin{equation}
\sup_{1 \leq sk \leq n-1} s^{-\nu} \left|  \beta_{k,n}(s)-B_{k,n}(s) \right|   =o_{P}(1)
\quad (0 \leq \nu <1/2).
\label{eq:qp_wc}
\end{equation}

Let $W$ denote a standard Wiener process on $[0,1]$.
The following result about the order of magnitude of $W$ near the origin, is important (see Remark 1 in \citealt{Csor:Maso:1989}): for $0 < \delta \leq 1/2$, we have
\begin{equation}
a^{-\delta}\sup_{0 \leq s \leq a}s^{-1/2+\delta}|B_{n}(s)|  \rightarrow 
\sup_{0 \leq s \leq 1}s^{-1/2+\delta}|W(s)|
\label{l:r_00}
\end{equation}
in distribution, as $a \downarrow 0$.
In addition, we make use of the following results on the empirical quantile function:
For $0 < \rho < \infty$, we have
(\citealt{Csor:Maso:1989},  Theorem 2.3 and  page 1454 )
\begin{eqnarray}
\label{l:r_0000}
&& \sup_{0< s < 1}{s}/{V_n(s)}, \ \sup_{\rho/n \leq s < 1}{V_n(s)}/{s}, \
\sup_{0< s < 1}{s}/{V_n^*(s)}, \
 \sup_{\rho/n \leq s < 1}{V_n^*(s)}/{s} \  = O_{P'}(1), \\
&&
\sup_{0< s \leq 1}\frac{ks}{(nV_{[ks],n})} , \
\sup_{1/k \leq s \leq 1}\frac{nV_{[ks],n}}{ks}, \
\sup_{0< s \leq 1}\frac{ks}{(nV^*_{[ks],n})}, \
\sup_{1/k \leq s \leq 1}\frac{nV^*_{[ks],n}}{ks} \ = O_{P'}(1).
\label{eq:08a}
\end{eqnarray}

\section{Preliminary lemmas to prove the theorems in the paper}

\begin{lemma}
\label{l:r_0}
For $0 \leq \nu <1/2$ and $\epsilon >0$, we have
\begin{equation}
 \sup_{0 \leq  s \leq 1}s^{-\nu}|B_{k,n}(s)|  = O_{P'}(1),
\quad \lim_{\delta \downarrow0} \liminf\limits_{n \rightarrow \infty} P'\left\{  \sup_{0 \leq  s \leq \delta}s^{-\nu}|B_{k,n}(s)| < \epsilon \right\} =1.
\label{eq:022}
\end{equation}
\end{lemma}
\begin{proof}
Let $\delta >0 $, $\mu= (1/2) - \nu,$ and $a = (k\delta/n).$
Then
 using the definition of $B_{k,n}$, we obtain,
 $\sup_{0 \leq    s \leq \delta}s^{-\nu}|B_{k,n}(s)|
   =  \delta^{\mu} a^{-\mu}
  \sup_{0 \leq   t \leq a} t^{-(1/2) + \mu}| B_n(t)|.$
Now, the proof follows from (\ref{l:r_00}).
\end{proof}

\begin{lemma}
\label{th:bqp_wc}
Let $0 \leq \nu <1/2$, $\mu= (1/2)-\nu$,  $0 < \lambda \leq 1$, and $0 < \delta < \min\{\mu, 1/4\}$.
Then
\begin{equation}
\sup_{\lambda/k \leq s \leq 1}  {  \left|  \beta_{k,n}^*(s)-B_{k,n}'(s) \right|}/{s^{\nu}} =O_{P'}(k^{-\delta}) = o_{P'}(1).
\label{eq:bqp_wc}
\end{equation}
\end{lemma}

\begin{proof}
For any $0<t \leq k/n$, we have $(n/k)^{\mu} t^{\mu-1/2} \leq (n/k)^{\delta} t^{\delta-1/2}$.
Then,
\begin{eqnarray*}
 &&\sup_{\lambda/k \leq s \leq 1}  {  \left|  \beta_{k,n}^*(s)-B_{k,n}'(s) \right|}/{s^{\nu}}
 = \sup_{\lambda/n \leq t \leq k/n} \left( n/k \right)^{\mu} {  \left|  \beta_{n}^*(t)-B_{n}'(t) \right|}/{t^{1/2-\mu}}\\
&\leq& \sup_{\lambda/n \leq t \leq k/n} n^{\delta} {  \left|  \beta_{n}^*(t)-B_{n}'(t) \right|}/{\{t(1-t)\}^{1/2-\delta}} k^{-\delta}  \sup_{1/n \leq t \leq k/n}(1-t)^{1/2-\delta}
= O_{P'}(k^{-\delta})= o_{P'}(1).
\end{eqnarray*}
where the ineqequality uses  (\ref{eq:bqp}).
\end{proof}


\begin{lemma}
\label{l:r_1}
Let $\epsilon >0$ be given.
Then, 
\begin{equation}
\sup_{0 < s \leq 1}  s^{1/2+\epsilon}
 \left|  k^{1/2}\big[ \log \{nk^{-1}V_{[ks],n}\} - \log \{ {nk^{-1}V_{[ks],n}^*}\} \big]
 - s^{-1}{B_{k,n}'(s)} \right|=o_{P'}(1).
\label{eq:logV_boot}
\end{equation}
\end{lemma}

\begin{proof}
We consider the supremum over $[1/k,1]$ and that over $(0, 1/k]$ separately.
First, we consider the former; therefore, let  $s \in [1/k, 1]$.
Let $\xi_l(s)$ and $\xi_u(s)$ denote the minimum and maximum of
$\{(n/k)V_{[ks],n}^*, (n/k)V_{[ks],n}\}$.
By the mean value theorem 
we have
\begin{equation}
k^{1/2} \big[ \log \{nk^{-1}V_{[ks],n}\} - \log \{ {nk^{-1}V_{[ks],n}^*}\}  \big]=
-{\xi_n}^{-1} k^{1/2}\left( nk^{-1}V_{[ks],n}  -  {nk^{-1}V_{[ks],n}^*}\right)= \xi_n^{-1}(s)\beta_{k,n}^*(s),
\notag
\end{equation}
where $\xi_n(s) \in [\xi_l(s),\xi_u(s)]$.
Then,  
$$\sup_{1/k \leq s  \leq 1}  s^{1/2+\epsilon}
 \left|  k^{1/2}\big[ \log \{nk^{-1}V_{[ks],n}\} - \log \{ {nk^{-1}V_{[ks],n}^*}\} \big]
 - s^{-1}{B_{k,n}'(s)} \right| \leq A_1 + A_2,$$
\begin{equation}
\label{eq-A1A2}
A_1 =   \sup_{1/k \leq s \leq 1}  \frac{s}{\xi_n(s)} \sup_{1/k \leq s \leq 1} \frac{ \left| \beta_{k,n}^*(s) - B_{k,n}'(s)\right|}{s^{1/2-\epsilon}},
\quad
A_2 = \sup_{1/k \leq s \leq 1} \frac{B_{k,n}'(s)}{s^{1/2-\epsilon} }   \frac{|\xi_n(s)-s|}{\xi_n(s)}.
\end{equation}
It may be verified using (\ref{l:r_00}), (\ref{l:r_0000}),
(\ref{eq:08a}), and Lemmas \ref{l:r_0} and \ref{th:bqp_wc},  that
$A_1 = o_{P'}(1)$ and $A_2 = o_{P'}(1).$
Next consider the proof for the supremum over
$0 < s < k^{-1}.$
First, substituting $s=1/k$ in the foregoing result for the supremum over $k^{-1} \leq s \leq 1$
we have
\begin{equation}
(1/k)^{1/2+\epsilon}  \left| k^{1/2}\big[ \log \{nk^{-1}V_{[ks],n}\} - \log \{ {nk^{-1}V_{[ks],n}^*}\} \big]
 - k{B_{k,n}'(1/k)} \right| = o_{P'}(1).
 \label{eq-5.3(a)}
\end{equation}
Then, the proof of the Lemma for the
interval $0 < s < 1/k$ follows  by (\ref{eq:bqp_wc}).
 %
\end{proof}


\begin{lemma}
\label{l:r_2}
Let $\epsilon>0$ be given.
Then,
 we have
$\sup_{\lambda/k \leq s \leq 1}  s^{\gamma +(1/2)+\epsilon}
 \left| L_n(s) - {B_{k,n}'(s)}/{s^{\gamma +1}} \right| =o_{P'}(1),$
 for $\gamma \in \mathbb{R}$ and $0 < \lambda  \leq 1$, where
$$L_n(s) = k^{1/2} \gamma^{-1}  \left\{   \left( {nk^{-1}V_{[ks],n}^*}\right)^{-\gamma} -1 \right\} -
  k^{1/2} {\gamma}^{-1} \left\{   \left( {nk^{-1}V_{[ks],n}}\right)^{-\gamma} -1 \right\}$$
%
and the function $\gamma^{-1}(s^{-\gamma}-1)$ is interpreted to be equal to $-\log s$ when $\gamma=0$.
\end{lemma}
\begin{proof}
For $\gamma=0$ the result is established in Lemma \ref{l:r_1}.
For $\gamma \neq 0$, we consider the suprema over $\lambda/k \leq s \leq 1/k $ and $ 1/k \leq s \leq 1 $ separately.
First, let us consider
$ 1/k \leq s \leq 1 $ with $\gamma \neq 0.$
Let $\xi_l(s)$ and $\xi_u(s)$ denote the minimum and maximum of
$\{(n/k)V_{[ks],n}^*, (n/k)V_{[ks],n}\}$
 ($1/k \leq s\leq1$).
Let $\varphi(t)= \gamma^{-1}({t^{-\gamma}-1})$.
Applying the mean value theorem to $\varphi(t)$ and noting that
$\dot{\varphi}(t)=-t^{-\gamma-1}$, we obtain
$L_n(s) =  \xi_n^{-\gamma-1}(s)\beta_{k,n}^*(s),$
%
where $\xi_n(s) \in [\xi_l(s),\xi_u(s)]$.
Then,
\begin{flalign*}
&\sup_{1/k \leq s \leq 1}  s^{\gamma +(1/2)+\epsilon}
 \left|L_n(s)
  - {B_{k,n}'(s)}/{s^{\gamma +1}} \right| \leq A_1 +A_2
\end{flalign*}
$$A_1 = \sup_{1/k \leq s \leq 1}  \left(\frac{ s} {\xi_n(s)}\right)^{\gamma +1} \sup_{1/k \leq s \leq 1} \frac{ \left| \beta_{k,n}^*(s) - B_{k,n}'(s)\right|}{s^{1/2-\epsilon}},$$
$$A_2 =
\sup_{1/k \leq s \leq 1} \frac{|B_{k,n}'(s)|}{s^{1/2-\epsilon} }
   \left| 1 - \{{s}/{\xi_n(s)}\}^{\gamma+1}\right|.$$
By (\ref{eq:08a}) and (\ref{eq:bqp_wc}),  $A_1=o_{P'}(1)$.
Next,
let $\delta \in (0,1).$  Choose $k > \delta^{-1}$.
Then,
\begin{flalign*}
 A_2
& \leq \sup_{0 < s \leq \delta} \frac{|B_{k,n}'(s)|}{s^{1/2-\epsilon} }
\sup_{1/k \leq s \leq 1}   \left| 1 - \{\frac{s}{\xi_n(s)}\}^{\gamma +1}\right|
 +  \sup_{0< s \leq 1} \frac{|B_{k,n}'(s)|}{s^{1/2-\epsilon} }
     \sup_{\delta < s \leq 1}   \left| 1 - \{\frac{s}{\xi_n(s)}\}^{\gamma +1}\right|  \\
& \leq B_1(n,k,\delta)B_2(n,k) + B_3(n,k)B_4(n,k, \delta), \qquad \mbox{say.}
\end{flalign*}
By using arguments similar to those for Lemma \ref{l:r_1}, it may be verified that $B_1(n,k,\delta)$ can be made arbitrarily small for large $\{n,k\}$ and small $\delta.$
Further, $B_2(n,k) = O_{p'}(1),$ and $B_4(n,k,\delta) = o_{P'}(1).$
Therefore, $A_2 = o_{P'}(1)$ and hence the proof follows for the range
$k^{-1} \leq s \leq 1.$
The proof for the supremum over $s \in [\lambda/k, 1/k]$ follows as in the proof of the previous Lemma.
This completes the proof of the Lemma.
\end{proof}

\section{Proofs of Theorems}
Let $V_1, \ldots, V_n$ be independent and identically distributed as Uniform(0,1) random variables,
 and $W_i=1-V_i, \ (i=1, \ldots n)$.
Let $X_i=F^{-1}(W_i), \ (i=1, \ldots n)$.
Then, $X_1, \ldots, X_n$ are independent and identically distributed random variables from the distribution  function $F$, and
$X_i = F^{-1}(1-V_i)$   ($i=1, \ldots n$).
Let $\{V_i'\}_{i=1}^n$ and $\{V_i^*\}_{i=1}^n$ be defined as at the beginning of Section \ref{sec:App-1}, more specifically as in
\cite{Csor:Maso:1989}.
Let
$X_i^*= F^{\leftarrow}(1-V_i^*)$ ($i=1, \ldots n$).
Since $\{V_1^*, \ldots, V_n^*\}\subseteq \{V_1,  \ldots, V_n\}$, it follows that $\{X_1^*, \ldots, X_n^*\}\subseteq \{X_1,  \ldots, X_n\}$.
Let $\{X_{1,n} \leq \cdots \leq X_{n,n}\}$ and $\{X_{1,n}^* \leq \cdots \leq X_{n,n}^*\}$ denote order statistics of $\{X_1,  \ldots, X_n\}$ and $\{X_1^*, \ldots, X_n^*\}$, respectively.
Then, since $U(x) = F^{\leftarrow}(1-x^{-1})$,
we have $X_{n-[ks]+1,n}=F^{\leftarrow}(1-V_{[ks],n}) =  U(1/V_{[ks],n})$ and
$X_{n-[ks]+1,n}^*=F^{\leftarrow}(1-V_{[ks],n}^*) =  U(1/V_{[ks],n}^*) \  (0 < ks \leq n)$.
Let $\cV_s = k/(n V_{[ks],n})$ and $\cV_s^* = k/(n V_{[ks],n}^*)$;
$\cV_s$ and $\cV_s^*$ also depend on $(k,n)$ but not shown explicitly.

\subsection{Proofs of main results in terms of joint probability $P'$}

Let $0< \tau <1$ be given, and let $\cT$ denote the interval
$\tau/k < s <  1 + (\tau/k).$
Throughout this section, suprema of functions of the real variable $s$ over the interval
$\cT(k, \tau)$  arises frequently.
To simplify notation we write $\sup_{s \in\cT} G(s)$ for  $\sup_{\tau/k < s < 1 + (\tau/k)} G(s)$
for any function $G(s)$.
Typically, the main results are proved for $\tau/k < s \leq 1$; in general, the proof could be extended to $\tau/k < s < 1 + (\tau/k)$ without much difficulty.
The next result is a bootstrap version of a result from \cite{Dree:1998} (see also
\citealt{Ha:Ferr:2006}, Theorem 2.4.2).
\begin{lemma}
\label{l:r_3}
Suppose that the conditions of Theorem \ref{th:mm_b_boot} are satisfied.
Let $\gamma$ denote the tail index of $F$ ($\gamma \in \mathbb{R}$).
Let $\epsilon>0$ and $0< \tau < 1$ be given.
Then, for suitably chosen function $a_0(\cdot)$, we have
\begin{equation}
\sup_{s \in \cT}  s^{\gamma+ (1/2)+\epsilon} \left| k^{1/2}
 \left\{ X_{n-[ks]+1,n}^*-X_{n-[ks]+1,n}  \right\}/{a_0({n}/{k})}   -  {s^{-\gamma-1}{B_{k,n}'(s)}}  \right| =  o_{P'}(1).
\label{6.3-a}
\end{equation}
\end{lemma}

\begin{proof}
Let
$\Lambda_n(s; \gamma, \rho,\delta)  =
\max\left\{\cV_s^{*(\gamma+\rho+\delta)}, \cV_s^{*(\gamma+\rho-\delta)}\right\}
 +
\max\left\{\cV_s^{\gamma+\rho+\delta}, \cV_s^{\gamma+\rho-\delta}\right\}$.
Let $\epsilon>0$ and  $ \delta>0$ be given. 
We apply inequality (2.3.17) in Theorem 2.3.6 of \cite{Ha:Ferr:2006}
with  $t := t_n=n/k$ and $x:= x_n(s) = k/(nV_{[ks],n})$ ($s \in \cT$).
Since  $\sup_{\cT} V_{[ks],n} \leq V_{k+1,n}= o_{P'}(1)$, we have that  $tx \geq \inf_{\cT} t_n x_n(s) \rightarrow \infty $, in probability[$P']$.
Hence, for any fixed $t_0$,  $t_n x_n(s)>t_0$ with probability approaching 1, uniformly on
 $s \in \cT$.
For the rest of the proof of this theorem, the inequalities are implicitly assumed to hold uniformly over $s \in \cT$, with probability approaching 1 as $n \rightarrow \infty,$ without further comment.

Let $a_0(t), A_0(t),A_*(t),$ and $a_*(t)$ be as in  Corollary 2.3.5 and Theorem 2.3.6 in \cite{Ha:Ferr:2006}.
Invoking the latter theorem, we
we obtain
\begin{equation}
 \left|\frac{U(1/V_{[ks],n})-U({n}/{k})}{a_0({n}/{k})} -
          \frac{\cV_s^{\gamma}-1}{\gamma}
          -  A_0\left({n}/{k}\right)  \Psi_{\gamma,\rho}(\cV_s) \right|
\leq    \varepsilon \left|A_0\left({n}/{k}\right)\right| \Lambda_n(s; \gamma, \rho,\delta).
\label{eq:3}
\end{equation}
Similarly, for the bootstrap sample, 
with $x_n(s)=k/(nV_{[ks],n}^*)$, we have
\begin{equation}
 \left|\frac{U(1/V_{[ks],n}^*)-U({n}/{k})}{a_0({n}/{k})}
- \frac{\left(\cV_s^*\right)^{\gamma}-1}{\gamma}  - A_0\left({n}/{k}\right)
    \Psi_{\gamma,\rho}\left(\cV_s^*\right) \right|
    \leq  \varepsilon  \left|A_0\left({n}/{k}\right)\right| \Lambda_n(s; \gamma, \rho,\delta).
\label{eq:4}
\end{equation}
Let
\begin{eqnarray*}
D_0(s) &=& s^{\gamma+1/2+\epsilon} \frac{B_{k,n}'(s)}{s^{\gamma}+1}, \quad
D_1(s) = s^{\gamma+1/2+\epsilon}k^{1/2}
\frac{X_{n-[ks]+1,n}^*-X_{n-[ks]+1,n}}{a_0({n}/{k})}\\
 D_2(s) &=& s^{\gamma+1/2+\epsilon}k^{1/2}
 \left\{ ({\cV_s^{*\gamma}-1})/{\gamma}-  ({\cV_s^{\gamma}-1})/{\gamma}\right\} \\
 D_3(s) &=& s^{\gamma+1/2+\epsilon}k^{1/2} A_0\left({n}/{k}\right)\left\{\Psi_{\gamma,\rho}\left(\cV_s^*\right)- \Psi_{\gamma,\rho}\left(\cV_s\right) \right\}\\
 D_4(s) &=& \varepsilon  k^{1/2}\left|A_0\left({n}/{k}\right)\right|
  s^{\gamma+(1/2)+\epsilon} \Lambda_n(s; \gamma, \rho,\delta).
\end{eqnarray*}

Subtracting (\ref{eq:3}) from (\ref{eq:4}), and multiplying by $k^{1/2}s^{\gamma+(1/2)+\epsilon}$,
we obtain,
$$|D_1(s) - D_2(s) - D_3(s) | \leq D_4(s) \quad  (s \in \cT).$$
By Lemma \ref{l:r_2},
for any $\gamma \in \mathbb{R}$, we have
$\sup_{s \in \cT} |D_2(s) - D_0(s)| = o_{P'}(1).$
%
Substituting  $\gamma=1$ in this inequality,
 we obtain $\sup_{s \in \cT}  s^{(3/2)+\epsilon}   k^{1/2} \left| \cV_s^* - \cV_s \right| =O_{P'}(1).$
By the mean value theorem,
$\Psi_{\gamma,\rho}(\cV_s^*)- \Psi_{\gamma,\rho}(\cV_s) = \dot{\Psi}_{\gamma,\rho}(\zeta_n(s))(\cV_s^*-\cV_s)$
where $\zeta_n(s)$ lies between $\cV_s^*$ and $\cV_s$.
By (\ref{eq:08a}), we have
\begin{eqnarray}
\sup_{s \in \cT} \ s\zeta_n(s) =  O_{P'}(1), \quad
\sup_{s \in \cT} \ \{ s\zeta_n(s)\}^{-1} =  O_{P'}(1),
\quad
\sup_{s \in \cT}  \ \log\{ s\zeta_n(s)\} = O_{P'}(1).
\label{eq-87d}
\end{eqnarray}
Now, using the functional form of $\Psi_{\gamma,\rho}(t)$ in (2.3.16) of \cite{Ha:Ferr:2006} on page 46,
 it may be verified that
 $\sup_{s \in \cT} | D_3(s)| = o_{P'}(1).$
Choose $\delta>0$ such that $\delta<1/2+\epsilon -\rho$.
By
(\ref{eq:08a}),
$\sup_{s \in \cT} (ks/(nV_{[ks],n}))^{\gamma+\rho\pm\delta} = O_{P'}(1)$.
Then,
\begin{equation}
\sup_{s \in \cT} s^{\gamma+1/2+\epsilon}
 \cV_s^{\gamma+\rho\pm\delta}
   = O_{P'}(1), \quad
 \sup_{s \in \cT} s^{\gamma+1/2+\epsilon}
   \cV_s^{*(\gamma+\rho\pm\delta)} = O_{P'}(1).
\label{eq:l54}
\end{equation}
By  (\ref{eq:l54}) and   $k^{1/2}A(n/k)=o(1)$, we obtain
$\sup_{s \in \cT} |D_4(s)| = o_{P'}(1).$
We have shown that
\newline $|D_1(s) - D_2(s) - D_3(s) | \leq D_4(s)$,
$\sup_{s \in \cT} |D_2(s) - D_0(s)| = o_{P'}(1)$,
$\sup_{s \in \cT} |D_3(s) = o_{P'}(1), $ and
 $\sup_{s \in \cT} |D_4(s)| = o_{P'}(1).$
Therefore, $\sup_{s \in \cT}  |D_1(s) - D_0(s)| = o_{P'}(1)$.
\end{proof}




%
The proof of the next lemma follows from the previous lemma,
(\ref{6.3-a}), and (\ref{eq:022}); therefore the proof is omitted.

\begin{lemma}
\label{l:r_4}
Suppose that the conditions of Lemma \ref{l:r_3} are satisfied,
and let $b_0(n) =  (X_{n,n}^* - X_{n,n}) I( \gamma<-1/2),$
where $I$ denotes the indicator function.
 Then, for $\epsilon >0$ and $0 < \tau  < 1$ we have
\begin{equation}
\sup_{0 < s \leq 1+\tau/k}  s^{\gamma+ (1/2)+\epsilon} \left| k^{1/2} \frac{X_{n-[ks]+1,n}^*-X_{n-[ks]+1,n}  - b_0(n)}{a_0(\frac{n}{k})}   -  \frac{B_{k,n}'(s)}{s^{\gamma+1}}  \right| =  o_{P'}(1).
\label{eq-N-6.13}
\end{equation}
\end{lemma}
\begin{lemma}
\label{l:r_3a}
Suppose that the conditions of Theorem \ref{th:mm_b_boot} are satisfied. Then,
\begin{equation}
k^{1/2}\{\hat b^*(n/k)-\hat b(n/k)\}/ \hat a(n/k) =  B_{k,n}'(1)+  o_{P'}(1).
 \label{eq:mm 3a}
 \end{equation}
\end{lemma}

\begin{proof}
Let $a_0$ be as in Theorem 2.3.6 in \cite{Ha:Ferr:2006}.
By the definitions of $\hat b^*(n/k)$ and $\hat b(n/k)$ , we have
\begin{eqnarray*}
k^{1/2}\frac{\hat b^*(n/k)-\hat b(n/k)}{ \hat a(n/k)}
&= &k^{1/2}\frac{X_{n-k,n}^*-X_{n-k,n}}{a_0(n/k)} \ \frac{a_0(n/k)}{a(n/k)} \ \frac{a(n/k)} { \hat a(n/k)}.
\end{eqnarray*}
By Lemma \ref{l:r_3} with $s=1$, we have
$
k^{1/2}\{{X_{n-k,n}^*-X_{n-k,n}}\}/{a_0(n/k)}  = B_{k,n}'(1)+o_{P'}(1).$
By  Lemma 4.7 in \cite{deHa:Resn:1993}, we have
 ${\hat{a}(n/k)}/{a(n/k)} = 1+ o_{P'}(1).$
By Theorem 2.3.6 and Corollary 2.3.5 in \cite{Ha:Ferr:2006},
  we have
${a_0(n/k)}/{a(n/k)} = 1+ o(1).$
Therefore, $k^{1/2}\{{\hat b^*(n/k)-\hat b(n/k)}\}/{ \hat a(n/k)} = $
$ [B_{k,n}'(1)+  o_{P'}(1)] [1+o_{P'}(1)] [1+o(1)] = B_{k,n}'(1)+  o_{P'}(1).$
\end{proof}

The next result assumes Condition B stated in the main paper.
%
%
%
%

\begin{proposition}
Let $X_1, X_2, \ldots$ be a given sequence of iid random variables with distribution function $F$, and suppose that $F$ satisfies Condition B.
Let $0 \leq \epsilon \leq 1/2$.
Then, for the sequence of Brownian bridges $B_n$  in
(\ref{eq:qp}),  we have
\begin{equation}
\label{brown-1}
\sup_{1/(n+1) \leq t \leq n/(n+1)}
n^{\epsilon}\{t(1-t)\}^{\epsilon-1/2}|n^{1/2}f\{Q(t)\}\big(Q(t) - X_{[nt],n}\big) - B_n(t) | = O_p(1)
\end{equation}
as $n \rightarrow \infty.$ [Note: \cite{Ha:Ferr:2006} use the notation $[x]$ for $\floor{x}.$]
\end{proposition}

\begin{proof}
The fact that there exists a sequence of Brownian bridges $B_n$ that satisfy (\ref{brown-1})
is the same as Theorem 6.2.1 in \cite{Cs-Horvath-1993}.
To show that the $B_n$ can be chosen to be the same as the one in  (\ref{eq:qp}) follows from the
 proof of this theorem given in \cite{Cs-Horvath-1993} on page 381.
The proof therein uses Theorem 4.2.1 which is the same as (\ref{eq:qp}).
\end{proof}

Let $\tilde{W}_{k,n}(s) = (n/k)^{1/2}B_n(1- ks/n)$,
where $B_n$ is the sequence of Brownian bridges that appear in (\ref{eq:qp}).
Also, let $C_n = \{ s :  n/(n+1) < ks <  (k+1)\}.$
The first part of the following lemma is the same as Lemma 2.4.10 of \cite{Ha:Ferr:2006}, and the second part follows  from the first.
\begin{proposition}
\label{lem-2.4.10}
Let $\epsilon >0$ and $\gamma \in \mathbb{R}$ be given.
Let $V$ be a uniform(0,1) random variable,
$\{V_1, V_2, \ldots \}$ be a given sequence of iid uniform (0,1) random variables, $Y =V^{-1}$, and
$Y_i = V_i^{-1}.$
Then we have the following:
\begin{equation}
\sup_{s \in C_n}
s^{\gamma+(1/2)+\epsilon}
 \left|k^{1/2}\frac{  \left\{ (k/n)Y_{n-\floor{ks},n})  \right\}^{\gamma} -1}{\gamma} -  k^{1/2}\frac{s^{-\gamma} - 1 } { \gamma} - \frac{\tilde{W}_{k,n}(s)  } {s^{\gamma+1}  }  \right| = o_{P'}(1)  \label{eq-2.4.10-a}
\end{equation}
\begin{equation}
\sup_{s \in C_n} s^{\gamma+(1/2)+\epsilon} \left| k^{1/2}\frac{  \left\{ (k/(nV_{[ks],n})  \right\}^{\gamma} -1}{\gamma} - k^{1/2} \frac{s^{-\gamma} - 1 } { \gamma} - \frac{\tilde{W}_{k,n}(s)  } {s^{\gamma+1}  }  \right| = o_{P'}(1).
\label{eq-2.4.10-b}
\end{equation}
%
\end{proposition}
\begin{proof}
The proof of (\ref{eq-2.4.10-a}) appears on page 53 of \cite{Ha:Ferr:2006} in their initial part of the proof of Lemma 2.4.10 therein, and it  is based on (\ref{brown-1}).
To prove (\ref{eq-2.4.10-b}), let
\begin{equation}
\Delta_{k,n}(s, \gamma) =
 k^{1/2}  \left( \big[ {  \left\{ k/(nV_{\floor{ks}+1,n})  \right\}^{\gamma} -1}\big] {\gamma}^{-1} -
\big[ {  \left\{ k/(nV_{[ks],n})  \right\}^{\gamma} -1}\big] {\gamma}^{-1} \right) \quad (0 < s \leq 1).
\label{eq:B.40}
\end{equation}
Then $\Delta_{k,n}(s, \gamma) = 0$ if $ s \not\in \{ 1/k, 2/k, \ldots \}.$
Let
$$A_{k,n}(s,\gamma) = s^{\gamma+(1/2)+\epsilon}
 \left[
k^{1/2}\frac{  \left\{ (k/(nV_{[ks],n})  \right\}^{\gamma} -1}{\gamma}
- k^{1/2} \frac{s^{-\gamma} - 1 } { \gamma}
-\frac{\tilde{W}_{k,n}(s)  } {s^{\gamma+1}  }\right].$$
Now,  substitute $Y_{n-\floor{ks}, n} = \{V_{\floor{ks}+1,n}\}^{-1}$ in (\ref{eq-2.4.10-a}),
 and add and subtract
$k^{1/2} \{{  \left\{ (k/(nV_{[ks],n})  \right\}^{\gamma} -1}\}/{\gamma}$
to the expression within the absolute sign in (\ref{eq-2.4.10-a}) to obtain
$\sup_{s \in C_n}
\left| A_{k,n}(s,\gamma) +\Delta_{k,n}(s, \gamma) \right| = o_{P'}(1).$
Since $\Delta_{k,n}(s, \gamma)= 0$ for $i/k < s < (i+1)/k$ ($i = 1, \ldots, k$),
it follows from the continuity properties of $A_{k,n}$ that
$\sup_{i/k < s \leq (i+1)/k} \left| A_{k,n}(s,\gamma) \right|
= \sup_{i/k < s < (i+1)/k} \left| A_{k,n}(s,\gamma) +\Delta_{k,n}(s, \gamma)\right|
 = o_{P'}(1).$
By a similar argument, $\sup_{\{n/(n+1)\}/k \leq s \leq 1/k} \left| A_{k,n}(s,\gamma) \right|=o_{P'}(1).$
Therefore,
$\sup_{s \in C_n} \left| A_{k,n}(s,\gamma) \right|
\leq \sup_{s \in C_n} \left| A_{k,n}(s,\gamma) +\Delta_{k,n}(s, \gamma)\right|
= o_{P'}(1)$, which is (\ref{eq-2.4.10-b}).
\end{proof}

\noindent \textbf{Remark:}
It can also be verified, using (\ref{eq-2.4.10-b}) that, for $0 < \lambda \leq 1$
\begin{equation}
\label{eq-Propo-6-2-c}
\sup_{\lambda < ks < k+1}
s^{\gamma+(1/2)+\epsilon} \left|
k^{1/2}\frac{  \left\{ (k/(nV_{[ks],n})  \right\}^{\gamma} -1}{\gamma}
- k^{1/2} \frac{s^{-\gamma} - 1 } { \gamma}
-\frac{\tilde{W}_{k,n}(s)  } {s^{\gamma+1}  }
  \right|
= o_{P'}(1).
\end{equation}



Let $\gamma_- = \min\{0, \gamma\}$.
As in the rest of this paper,  the function $(s^{-\gamma}-1)/ \gamma$ is interpreted as $-\log s$ when $\gamma=0$.

\begin{lemma}
\label{l:mm01}
Suppose that the conditions of Theorem \ref{th:mm_gamma_boot} are satisfied.
Then, for $\epsilon >0$,  we have
\begin{equation}
\sup_{s \in C_n}  s^{\gamma_-+ 1/2+\epsilon}
\Big| k^{1/2} \frac{\log X_{n-[ks]+1,n}  -\log U(\frac{n}{k})}{q_0(\frac{n}{k})}
 -  k^{1/2} \frac{s^{-\gamma}-1}{\gamma}
 - \frac{\tilde{W}_{k,n}(s) }  {s^{\gamma+1}  } \Big| =  o_{P'}(1)
   \label{eq:mm l1.1}
\end{equation}
\begin{equation}
 \sup_{s \in C_n}  s^{\gamma_-+ 1/2+\epsilon} \Big| k^{1/2}
  \frac{ \log X_{n-[ks]+1,n}^*  - \log U(\frac{n}{k})}{q_0(\frac{n}{k})}
  -  k^{1/2} \frac{s^{-\gamma}-1}{\gamma}
  - \frac{ \tilde{W}_{k,n}(s)+ B_{k,n}'(s)}{s^{\gamma+1}}  \Big|
    =  o_{P'}(1)
\label{eq:mm l1.2}
\end{equation}
\begin{equation}
 \sup_{s \in C_n}  s^{\gamma_-+ 1/2+\epsilon} \Big| k^{1/2}
\frac{ \log X_{n-[ks]+1,n}^*  - \log X_{n-[ks]+1,n} } {q_0(\frac{n}{k})}    - \frac{B_{k,n}'(s)}{s^{\gamma_-+1}}  \Big| =  o_{P'}(1).
 \label{eq:mm l1.3}
\end{equation}
%

\end{lemma}
%
%
%

\begin{proof}
\textit{Proof of (\ref{eq:mm l1.1}):}
Let $ \delta>0$.
Recall $\cV_s = k/(n V_{[ks],n})$ and $\cV_s^* = k/(n V_{[ks],n}^*)$;
$\cV_s$ and $\cV_s^*$ also depend on $(k,n)$ but not shown explicitly.
We apply inequality at the top of page 104 in \cite{Ha:Ferr:2006}
 with
$t := t_n=n/k$ and $x:= x_n(s) = k/(nV_{[ks],n})$  ($s \in C_n$).
Since  $\sup_{s \in C_n} V_{[ks],n} \leq V_{k+1,n}= o_{P'}(1)$,
$tx \geq \inf_{s \in C_n} t_n x_n(s) \geq \{V_{k+1,n}\}^{-1}  \rightarrow \infty $, in probability[$P']$.
Hence, for any fixed $t_0$,  $t_n x_n(s)>t_0$ with probability approaching 1, uniformly on
 $s \in C_n$.
 For the rest of the proof of this theorem, the inequalities are implicitly assumed to hold uniformly over $s \in C_n$, with probability approaching 1 as $n \rightarrow \infty,$ without further comment.
Let $A_2(s) = (\cV_s^{\gamma_-}-1)/\gamma_-$,
$A_3(s) = Q_0\left({n}/{k}\right)  \Psi_{\gamma_-,\rho'}\left(\cV_s\right)$, and
$A_4(s) = \varepsilon \left|Q_0\left({n}/{k}\right)\right| \cV_s^{\gamma_-+\rho'}\max\left\{\cV_s^\delta, \cV_s^{-\delta}\right\}.$
%
Applying the inequality at the top of page 104 in \cite{Ha:Ferr:2006},
with the foregoing
 choices for $t$ and $x$,
and then multiplying the result by $k^{1/2}s^{\gamma_-+(1/2)+\epsilon}$, we obtain,
\begin{equation}
\begin{split}
& s^{\gamma_-+1/2+\epsilon}k^{1/2} \left|\frac{\log X_{n-[ks]+1,n}-\log U(\frac{n}{k})}{q_0(\frac{n}{k})}
- A_2(s) - A_3(s) \right|
\leq  s^{\gamma_-+1/2+\epsilon}k^{1/2} |A_4(s)|
\end{split}
\label{eq:r_2}
\end{equation}
Let
$D_{k,n}(s, \gamma) =   \big[ k^{1/2} \{{s^{-\gamma}-1}\}/{\gamma} + {\tilde{W}_{k,n}(s) } / {s^{\gamma+1}  }\big]$.
Since (\ref{eq-2.4.10-b}) holds or any $\gamma \in \mathbb{R}$, we have
$\sup_{s \in C_n}
s^{\gamma_-+(1/2)+\epsilon}
\left| k^{1/2} A_2(s) - D_{k,n}(s, \gamma_-)   \right| =  o_{P'}(1).$
Using the  form of  $\Psi_{\gamma,\rho}$ in Corollary 2.3.5 of
\cite{Ha:Ferr:2006} on page 46,
it may be verified that
$\sup_{s \in C_n}  s^{\gamma_-+ (1/2)+\varepsilon} k^{1/2} |A_3(s) | = o_{P'}(1)$.
Using(\ref{eq:l54}) and Assumption (A.4), it may be verified that
$\sup_{s \in C_n}  s^{\gamma_-+ (1/2)+\varepsilon} k^{1/2} |A_4(s) | = o_{P'}(1)$.
The result follows by substituting these in (\ref{eq:r_2}).
The proof of the next part is similar to that of the previous one, and hence omitted.
\end{proof}



For $\gamma \geq -(1/2),$ let $b_0(n) =b_0^*(n)= \log U({n}/{k})$;
for $\gamma< -(1/2)$,
let $$b_0(n) = \log U(1/V_{1,n}) + \gamma_-^{-1} {q_0(n/k)}, \
b_0^*(n) =  \log U(1/V_{1,n}^*) + \gamma_-^{-1}{q_0(n/k)}.$$
%


The statement of the two results in the next lemma, and their proofs are similar to those of
the previous Lemma \ref{l:mm01}.  Therefore, the proof is omitted.

\begin{lemma}
\label{l:mm01a}
Suppose that conditions of Theorem \ref{th:mm_gamma_boot} are satisfied. Then,
for $\epsilon >0$,
 we have
\begin{equation}
\sup_{0 < ks < k+1}  s^{\gamma_-+ 1/2+\epsilon} \left| k^{1/2} \frac{ \log X_{n-[ks]+1,n}  - b_0(n)} {q_0(\frac{n}{k})}
  - k^{1/2}\frac{s^{-\gamma_-}-1}{\gamma_-}
  - \frac{\tilde{W}_{k,n}(s)}{s^{\gamma_-+1}}
  \right| =  o_{P'}(1),
 \label{eq:mm c1.1}
\end{equation}
\begin{equation}
\sup_{0 < ks < k+1}  s^{\gamma_-+ 1/2+\epsilon} \left| k^{1/2} \frac{ \log X_{n-[ks],n}^*  - b_0^*(n)} {q_0(\frac{n}{k})}
+ k^{1/2}\frac{s^{-\gamma_-}-1}{\gamma_-}  - \frac{\tilde{W}_{k,n}(s) + B_{k,n}'(s)}{s^{\gamma_-+1}} \right| =  o_{P'}(1).
 \label{eq:mm c1.2}
\end{equation}
\end{lemma}

\begin{lemma}
\label{l:mm02}
Suppose that the conditions of Theorem \ref{th:mm_gamma_boot} are satisfied.
 Then, for  $\epsilon >0$, we have
\begin{equation}
\begin{aligned}
   k^{1/2} & \frac{ \log X_{n-[ks]+1,n}  - \log X_{n-k,n}} {q_0(\frac{n}{k})}   -  k^{1/2} \frac{s^{-\gamma_-}-1}{\gamma_-}   - \frac{\tilde{W}_{k,n}(s)}{s^{\gamma_- + 1}}  + \tilde{W}_{k,n}(1) \\ &
\qquad
= {(1+s^{-(\gamma_-+ 1/2+\epsilon)})} o_{P'}(1),
\end{aligned}
 \label{eq:mm l2.1}
\end{equation}
\begin{equation}
\label{eq:mm l2.2}
\begin{aligned}
 k^{1/2} & \frac{ \log X_{n-[ks]+1,n}^*  -\log X_{n-k,n}^*} {q_0(\frac{n}{k})}   -  k^{1/2} \frac{s^{-\gamma_-}-1}{\gamma_-} - \frac{\tilde{W}_{k,n}(s) + B_{k,n}'(s)}{s^{\gamma_-+1}}   \\ &
\qquad  + \tilde{W}_{k,n}(1) + B'_{k,n}(1)  = { (1+s^{-(\gamma_-+ 1/2+\epsilon)}) } o_{P'}(1),
\end{aligned}
\end{equation}
\begin{equation}
\label{eq:mm l2.3}
\begin{aligned}
  k^{1/2} & \frac{ \log X_{n-[ks]+1,n}^*  -\log X_{n-k,n}^*-    (\log X_{n-[ks]+1,n}  - \log X_{n-k,n})} {q_0(\frac{n}{k})} - \frac{B_{k,n}'(s)}{s^{\gamma_-+1}} \\ &
\qquad {  + B'_{k,n}(1) }
 = { (1+s^{-(\gamma_-+ 1/2+\epsilon)})} o_{P'}(1)
\end{aligned}
\end{equation}
where $o_{P'}(1)$ in (\ref{eq:mm l2.1}), (\ref{eq:mm l2.2}), and (\ref{eq:mm l2.3})
do not depend on $s$ ($0 < ks < k+1$).
\end{lemma}

\begin{proof}
It follows from (\ref{eq:mm c1.1}) that
\begin{equation}
\sup_{k < ks < k+1}  s^{\gamma_-+ 1/2+\epsilon} \left| k^{1/2} \frac{ \log X_{n-k,n}  - b_0(n)} {q_0(\frac{n}{k})}
- k^{1/2}\frac{s^{-\gamma_-}-1}{\gamma_-}   - \frac{\tilde{W}_{k,n}(s)}{s^{\gamma_-+1}}
  \right| =  o_{P'}(1),
\label{eq:mm c1.1b}
\end{equation}
By a continuity argument similar to that for (\ref{lem-2.4.10}),
(\ref{eq:mm c1.1b}) also holds with $\sup_{k < ks < k+1}$ replaced by $\sup_{k \leq ks < k+1}$
and hence at $s=1$,
%
%
from which  we obtain
$$    k^{1/2} b_0(n)/{q_0(\frac{n}{k})}  =  k^{1/2} \log X_{n-k,n}/{q_0(\frac{n}{k})}
-  \tilde{W}_{k,n}(1) +  o_{P'}(1).$$
Substituting this in ( \ref{eq:mm c1.1}) leads to
(\ref{eq:mm l2.1}).
By a similar argument applied to (\ref{eq:mm c1.2}), we have
\begin{equation}
 k^{1/2} \{ \log X_{n-k,n}^*  - b_0^*(n)\}/ {q_0(\frac{n}{k})}  - \{\tilde{W}_{k,n}(1)+B_{k,n}'(1)\} =  o_{P'}(1).
 \label{eq:mm c1.1e}
 \end{equation}
Substitution of this in (\ref{eq:mm c1.2}) leads to (\ref{eq:mm l2.2}).
The third part  follows from the first two.
\end{proof}



The asymptotic distributions of various tail statistics turn out to be functionals of a Wiener process.
Let $W(s)$, $0\leq s \leq 1$, denote a standard Wiener process, and let
us define the following random variables that are functionals of the same $W$:
\begin{eqnarray}
 &&Q =  2\int_0^1 \frac{s^{-\gamma}-1}{\gamma_{-}}\left(\frac{W(s)}{s^{\gamma_-+1}}-W(1) \right)ds,
\qquad \qquad \qquad  P = \int_0^1 \frac{W(s)}{s^{\gamma_-+1}}-W(1) ds,  
    \label{eq-funtional-Q}\\
&& R =  \left(1-\gamma_{-}\right)^2(1-2 \gamma_{-}) \left( 2^{-1}{(1-2 \gamma_{-})}{Q}-2{P} \right),
 \qquad \qquad  {\Gamma} = {\gamma_+} {P} + {R}
\label{eq-tildeR_n} \\
 && A =  \gamma_+ W(1) + (1-\gamma_{-})(3-4 \gamma_{-})P - 2^{-1}(1-\gamma_{-})(1-2 \gamma_{-})^2Q
\label{eq-funtional-}
\end{eqnarray}
These random variables will appear in the limiting distributions of various statistics.
In finite samples, the following random variables that correspond these will appear:
Let $(P_n', Q_n', A_n', \Gamma_n')$ be defined in the same way except that the standard Wiener process $W$ is replaced by $B_{k,n}'$.
Similarly, let $(\tilde{P}_n, \tilde{Q}_n, \tilde{A}_n, \tilde{W}_n)$ be defined in the same way except that $W$ is replaced by $\tilde{W}_{k,n}$, where
$\tilde{W}_{k,n}(s) = (n/k)^{1/2}B_n(1- ks/n)$.

%
Since $\tilde{W}_{k,n}(s)$
and $B_n'(ks/n)$ are equal in distribution, it follows that
$\tilde{P}_n \stackrel{d}{=} P_n'$, but they are not the same random variables.
Similar comments apply to the other pairs of random variables of the form
 $(\tilde{Y}_{k,n}, Y_{k,n}')$,  with different symbols for $Y$, that will appear in the derivations.

\begin{lemma}
\label{l:mm03}
\begin{eqnarray}
 k^{1/2}\left\{  {H_n}/{q_0(n/k)} - (1-\gamma_-)^{-1} \right \} &=& \tilde{P}_n +o_{P'}(1),
 \label{eq:mm l3.1}\\
 k^{1/2}\left\{  {H_n^*}/{q_0(n/k)} - (1-\gamma_-)^{-1} \right \} &=& (P_n'+ \tilde{P}_n) +o_{P'}(1),
 \label{eq:mm l3.2} \\
 k^{1/2}\left\{  (H_n^*-H_n)/q_0(n/k) \right \} &=&  P_n' +o_{P'}(1).
 \label{eq:mm l3.3}
\end{eqnarray}
%
\end{lemma}
\begin{proof}
The proof of (\ref{eq:mm l3.1}) follows from (\ref{eq:mm l2.1}); (\ref{eq:mm l3.2}) follows by similar arguments from (\ref{eq:mm l2.2}).
Finally, (\ref{eq:mm l3.3}) follows from the previous two parts.
\end{proof}

\begin{lemma}
\label{l:mm04}
\begin{equation}
 k^{1/2}(H_n - \gamma_+) = \gamma_+ \tilde{P}_n +o_{P'}(1), \qquad
  k^{1/2}(H_n^* - H_n ) = \gamma_+ P_n' +o_{P'}(1).
 \label{eq:mm l4.1}
\end{equation}
\end{lemma}

\begin{proof}
A proof of the first part of (\ref{eq:mm l4.1})  is given on page 109 in \cite{Ha:Ferr:2006}.
To  prove the second part, it follows from (\ref{eq:mm l3.3}) that
 $k^{1/2}(  H_n^*-H_n ) =  q_0(n/k)[P_n' + o_{P'}(1)]$.
%
Since $\{q_0(t) - \gamma_+\}/{Q(t)} = O(1)$ as $t \rightarrow \infty$ [see \citealt{Ha:Ferr:2006}, page 109], it follows from $k^{1/2}Q(t) = o(1)$ (see Assumption (A.4)) that
$\{q_0(n/k) - \gamma_+\} = \big[ \{q_0(n/k) - \gamma_+\}/{Q(n/k)}\big] Q(n/k)= O(1)O(k^{-1/2}) = o(1)$.
Now, the proof follows easily.
%
\end{proof}


\begin{lemma}
\label{l:mm05}
\begin{equation}
 k^{1/2}\left(  \left(\frac{H_n}{q_0(n/k)}\right)^2 - \left(\frac{1}{1-\gamma_-}\right)^2 \right ) = \frac{2}{1-\gamma_-} \tilde{P}_n + o_{P'}(1),
 \label{eq:mm l5.1}
\end{equation}
\begin{equation}
 k^{1/2}\left(  \frac{(H_n^*)^2-H_n^2}{q_0^2(n/k)} \right ) = \frac{2}{1-\gamma_-} P_n' + o_{P'}(1).
 \label{eq:mm l5.3}
\end{equation}
\end{lemma}

\begin{proof}
Proof of (\ref{eq:mm l5.1}) follows from (\ref{eq:mm l3.1}) by an application of mean value theorem;
similarly, (\ref{eq:mm l3.3}) follows from (\ref{eq:mm l3.3}).
By similar arguments, we may also show the following, but it is not directly required:
\begin{equation}
 k^{1/2}\left(  \left(\frac{H_n^*}{q_0(n/k)}\right)^2 - \left(\frac{1}{1-\gamma_-}\right)^2 \right ) = \frac{2}{1-\gamma_-} (\tilde{P}_n +P_n') + o_{P'}(1),
 \label{eq:mm l5.2}
\end{equation}
\end{proof}


%
%

Let
\begin{equation}
\lambda_1 = \frac{1}{1-\gamma_{-}}, \qquad
\lambda_2 = \frac{2}{\left(1-\gamma_{-}\right)\left(1-2 \gamma_{-}\right)}.
\label{eq-48.1}
\end{equation}

\begin{lemma}
\label{l:mm06}
\begin{equation}
 k^{1/2}\left( \{q_{0}\left({n}/{k}\right\}^{-2}{M_{n}} - \lambda_2 \right) - \tilde{Q}_n = o_{P'}(1),
 \label{eq:mm l6.1}
\end{equation}
\begin{equation}
 k^{1/2}\left(\{q_{0}\left({n}/{k}\right\}^{-2} {M_{n}^*} -  \lambda_2 \right) - (\tilde{Q}_n + Q_n')= o_{P'}(1),
 \label{eq:mm l6.2}
\end{equation}
\begin{equation}
 k^{1/2} \{q_{0}\left({n}/{k}\right\}^{-2} (M_{n}^*-M_n)- Q_n' = o_{P'}(1).
 \label{eq:mm l6.3}
\end{equation}
\end{lemma}

\begin{proof}
Let $0< \epsilon <1/4$.
By (\ref{eq:mm l2.1}), for $0 <  s < 1$,
\begin{eqnarray}
&&\frac{  \log X_{n-[ks]+1,n}  - \log X_{n-k,n}}{q_o(n/k)} \notag \\
 &=&   \frac{s^{-\gamma_-}-1}{\gamma_-} + k^{-1/2}\left( \frac{\tilde{W}_{k,n}(s)}{s^{\gamma_-+1}}
 - \tilde{W}_{k,n}(1)\right)
  + k^{-1/2} {(1+s^{-1/2-\epsilon-\gamma_-})}o_{P'}(1) \notag \\
 &=&   f_{1n}(s) + f_{2n}(s) + f_{3n}(s), \quad \mbox{say}. \label{eq-6.72AA}
\end{eqnarray}
Then
\begin{eqnarray}
k^{1/2} \frac{M_n}{\left(q_o(n/k)\right)^2}
 &=&   \int_0^1 k^{1/2}\{f_{1n}(s) + f_{2n}(s) + f_{3n}(s) \}^2 \ ds. \label{eq-M-2}
\end{eqnarray}
Note that if $2\{(1/2)+\epsilon+\gamma_-\} \geq 1$ then
$\int_0^1 f_{3n}^2(s) \ ds $ is not finite.
Therefore, though the integral in (\ref{eq-M-2}) is finite,
 it is not possible to evaluate it by expanding the squared expression under the integral sign in (\ref{eq-M-2}) and
 evaluating the integral of each resulting term.
Therefore, we express $\int_0^1 $ as $\int_0^{1/k} + \int_{1/k}^1$, and evaluate the two integrals separately.
This is equivalent to evaluating the first term of $M_n$ and the rest separately.
Let
$k^{1/2}M_n /q_0^2(n/k) = A_1(k,s) + A_2(k,n)$ where
\begin{eqnarray}
A_1(k,n) &= &  \big[ k^{-1/4} \{\log X_{n,n}  - \log X_{n-k,n}\}/q_0(n/k) \big]^2\\
A_2(k,n) &=& \sum_{i=1}^{k-1} \big[ k^{-1/4} \{\log X_{n-i,n}  - \log X_{n-k,n}\}/q_0(n/k)\big]^2.
\end{eqnarray}
It follows from (\ref{eq-6.72AA}) with $s= 1/k$ that
\begin{eqnarray*}
A_1(k,n) &=&  \big[ k^{-1/4} f_{1n}(k^{-1}) + k^{-1/4} f_{2n}(k^{-1})+ k^{-1/4} f_{3n}(k^{-1})  \big]^2 = o_{P'}(1).
\end{eqnarray*}
Next, note that
$A_2(k,n) = \int_0^{1/k} k^{1/2}\{ f_{1n}(s) + f_{2n}(s) + f_{3n}(s) \}^2 \ ds = o_{P'}(1).$
Expanding the integrand and evaluating each resulting integral, it may be verified that
\begin{equation*}
k^{1/2} \frac{M_n}{\left(q_o(n/k)\right)^2} =  \frac{2k^{1/2}}{\left(1-\gamma_{-}\right)\left(1-2 \gamma_{-}\right)} +  2\int_{0}^1 \frac{s^{-\gamma}-1}{\gamma_{-}}\left(\frac{\tilde{W}_{k,n}(s)}{s^{\gamma_-+1}}-\tilde{W}_{k,n}(1) \right)ds +  o_{P'}(1).
\end{equation*}
This completes the proof of (\ref{eq:mm l6.1}).
The proof of (\ref{eq:mm l6.2}) is similar, and hence omitted.
The proof of (\ref{eq:mm l6.3}) follows by subtracting (\ref{eq:mm l6.1}) from (\ref{eq:mm l6.2}).
\end{proof}


\begin{lemma}
\label{l:mm07}
\begin{equation}
 k^{1/2} \left(\frac{H_n^2}{M_n} - \frac{\lambda_1^2}{\lambda_2}   \right) =  (1-2\gamma_-) \tilde{P}_n - \frac{\left(1-2 \gamma_{-}\right)^2}{4} \tilde{Q}_n + o_{P'}(1).
 \label{eq:mm l7.1}
\end{equation}
\begin{equation}
 k^{1/2} \left(\frac{(H_n^*)^2}{M_n^*}  - \frac{H_n^2}{M_n}\right) =  (1-2\gamma_-) P_n' - \frac{\left(1-2 \gamma_{-}\right)^2}{4} Q_n' + o_{P'}(1).
 \label{eq:mm l7.3}
\end{equation}
\end{lemma}

\begin{proof}
Consider the left hand side of (\ref{eq:mm l7.1}):
\begin{eqnarray}
&&k^{1/2} \left(\frac{H_n^2}{M_n} - \frac{\lambda_1^2}{\lambda_2}  \right)
 =   \frac{\lambda_1^2}{\lambda_2}\frac{q_0^2(n/k)}{M_n} k^{1/2} \left( \lambda_2-\frac{ M_n } {q_0^2(n/k)} \right)  +   \frac{q_0^2(n/k)}{M_n} k^{1/2} \left( \frac {H_n^2  }{q_0^2(n/k)}-\lambda_1^2  \right) \label{eq-49.1}
\end{eqnarray}
The first part follows from  (\ref{eq:mm l6.1}) and (\ref{eq:mm l5.1}).
The proof of the second part is similar.
\end{proof}


%
%

\begin{lemma}
\label{l:mm09}
\begin{equation}
k^{1/2}(\hat{\gamma}_- - \gamma_-) =  \tilde{R}_n +o_{P'}(1), \qquad
k^{1/2}(\hat{\gamma}^*_- - \hat{\gamma}_-) =  R_n' +o_{P'}(1), \quad
k^{1/2}(\hat{\gamma}^* - \hat{\gamma}) =  \Gamma_n' +o_{P'}(1).
 \label{eq:mm l9.3}
\end{equation}
\end{lemma}

\noindent\textit{Proof}.
It follows from (\ref{eq-48.1}) that
$\gamma_- = 1 -  2^{-1}\{1-{\lambda_1^2}/{\lambda_2}\}.$
Then,  by the definition of $\hat{\gamma}_-$,
\begin{eqnarray}
k^{1/2}( \hat{\gamma}_- - \gamma_-) 
& = & 2^{-1}{k^{1/2}} \left( {\lambda_1^2}/{\lambda_2} -   {H_n^2}/{M_n} \right)
\left(  1-{\lambda_1^2}/{\lambda_2}  \right)^{-1}
 \left( 1-H_n^2/M_n \right)^{-1}. \label{eq-52.1}
\end{eqnarray}
Now, the first part of the lemma may be verified using (\ref{eq:mm l7.1}).
The proof of the second part is similar and hence omitted.
 Using the definitions of $\hat{\gamma}^*$ and $\hat{\gamma}$, and Lemmas \ref{l:mm09} and \ref{l:mm04}, we obtain
\begin{eqnarray*}
k^{1/2}(\hat{\gamma}^* - \hat{\gamma}) &=& k^{1/2}(\hat{\gamma}^*_- - \hat{\gamma}_-) + k^{1/2}(\hat{\gamma}^*_+ - \hat{\gamma}_+)
 =  R_n' + \gamma_+P_n' + o_{P'}(1)= \Gamma_n' + o_{P'}(1). \hfill \blacksquare
\end{eqnarray*}

\begin{lemma}
\label{l:mm11}
\begin{equation}
k^{1/2}\left\{{\hat{a}^*(n/k)}/{\hat{a}(n/k)} - 1\right\} =A_n' +o_{P'}(1).
 \label{eq:mm l11}
\end{equation}
\end{lemma}

\begin{proof}
Using the definitions, we have
\begin{eqnarray}
k^{1/2}\left(\frac{\hat{a}^*(n/k)}{\hat{a}(n/k)} - 1\right) &=& k^{1/2}
\left( \frac{X_{n-k,n}^*}{X_{n-k,n}} \  \frac{H_{n}^*}{H_{n}} \ \frac{(1-\hat{\gamma}^*_-)}{(1-\hat{\gamma}_-)}  -1\right).
\label{eq-4.14.1a}
\end{eqnarray}

To prove the lemma, we substitute asymptotic representations for the
 three terms $  X_{n-k,n}^*/X_{n-k,n}$, $H_{n}^*/H_{n} $, and $ (1-\hat{\gamma}^*_-)/(1-\hat{\gamma}_-)$ in (\ref{eq-4.14.1a}).
%
By (\ref{eq:mm l1.3}) and an application of mean value theorem, we obtain
$k^{1/2} \{q_0({n}/{k})\}^{-1} \log \{X_{n-k,n}^*/X_{n-k,n} \}    =  B_{k,n}'(1) + o_{P'}(1).$
Since $q_0(n/k) = \gamma_+ +  o(1)$,
 it may be verified that ${X_{n-k,n}^*}/{X_{n-k,n}} = 1 + k^{-1/2}\{\gamma_+ B_{k,n}'(1) + o_{P'}(1)\}.$
Hence,
\begin{equation}
  \frac{X_{n-k,n}^*}{X_{n-k,n}} = 1 + k^{-1/2}\{\gamma_+ B_{k,n}'(1)  + o_{P'}(1)\}.
\label{eq:mm l11b}
\end{equation}
We adopt a similar procedure for  $\left(H_{n}^*/H_{n}\right)$.
First, we write
\begin{equation}
k^{1/2}\left(\frac{H_{n}^*}{H_{n}} -1 \right) = (1-\gamma_-)k^{1/2} \frac{H_{n}^*-H_{n}}{q_0(n/k)} \frac{q_0(n/k)}{(1-\gamma_-)H_n}.
\label{eq-14.1f}
\end{equation}
Then, it may be verified, using Lemma  \ref{l:mm03}, that
$ ({H_{n}^*}/{H_{n}}) = 1 + k^{-1/2}\{(1-\gamma_-)P_n'   + o_{P'}(1)\}.$
%
Using Lemma \ref{l:mm09}, we may verify that
%
$({1-\hat{\gamma}^*_-})/({1-\hat{\gamma}_-}) = 1 + k^{-1/2}(-R_n'(1-\gamma_-)^{-1}  + o_{P'}(1)).$
Now the lemma follows by substituting the foregoing asymptotic representations
 in  (\ref{eq-4.14.1a}).
\end{proof}

\subsection{Estimation of high quantile}

In this subsection we establish the main results for constructing a confidence interval for a high quantile.
Let $p_n$ be a given small number in the range $0 < p_n < 1$, and let
 $x(p_n) = F^{\leftarrow}(1-p_n)$ denote the unknown upper $p_n$th quantile of $F$.
Recall that $d_n = k/(np_n)$.
An estimator of $x(p_n)$ is
(see, \citealt{Ha:Ferr:2006},  Theorem 4.3.1 )
$$\hat{x}(p_n)  = \hat{b}\left({n}/{k}\right) + \hat{a}\left({n}/{k}  \right) \{d_n^{\hat\gamma}-1\}/{\hat\gamma};$$
the bootstrap counterpart of $\hat{x}(p_n) $ is
$\hat{x}^*(p_n)  = \hat{b}^*\left({n}/{k}\right) + \hat{a}^*\left({n}/{k}  \right) \{d_n^{\hat{\gamma}^*}-1\}/{\hat\gamma}^*.$
Let $ q_{\gamma}(t) = \int_1^t s^{\gamma-1}\log s\ ds,$  $( t>0).$
The asymptotic distribution of the bootstrap quantity $\hat{x}^*(p_n)$ is provided in the next lemma;
 it can be used for constructing a confidence interval for the high-quantile $x(p_n).$

\begin{lemma}
\label{l:mm12}
Suppose that the conditions of Theorem \ref{th:mm_x_boot} are satisfied.
Then,
\begin{eqnarray}
\frac{k^{1/2}}{q_{\hat\gamma}(d_n)}  \frac{\hat{x}^*(p_n)-\hat{x}(p_n)}{\hat{a}\left({n}/{k}  \right)}
 &=& \Gamma_n' + (\gamma_-)^2 B_{k,n}'(1) - (\gamma_-)A_n' +  o_{P'}(1).
\label{eq:r_bq_wc} \\
\frac{k^{1/2}}{q_{\hat{\gamma}^*}(d_n)}  \frac{\hat{x}^*(p_n)-\hat{x}(p_n)}{\hat{a}
\left({n}/{k}  \right)}
 &=&  \Gamma_n' + (\gamma_-)^2 B_{k,n}'(1) - (\gamma_-)A_n' +  o_{P'}(1).
\label{eq:r_bq_wc-2}
\end{eqnarray}
\end{lemma}

\begin{proof}
To simplify notation, we write $\hat{x}_n$ and
$\hat{x}^*_n$ for $\hat{x}^*(p_n)$ and  $\hat{x}(p_n)$ respectively;
this simplified notation is applicable only to this proof.
Starting from the definitions of these quantities, we have
\begin{eqnarray}
 k^{1/2}  \frac{\hat{x}_n^*-\hat{x}_n}{q_{\hat\gamma}(d_n)\hat{a}\left({n}/{k}  \right)}
&=&  k^{1/2}  \frac{\hat b^*\left({n}/{k}  \right)-\hat b\left({n}/{k}  \right)} {q_{\hat\gamma}(d_n)\hat{a}\left({n}/{k}  \right)}
+  k^{1/2} \left( \frac{\hat{a}^*\left({n}/{k}  \right) } {\hat{a}\left({n}/{k}  \right)}   -1  \right)  \frac{d_n^{\hat\gamma}-1}{\hat\gamma q_{\hat\gamma}(d_n)}  \notag\\
 + &&
\frac{k^{1/2}}{q_{\hat\gamma}(d_n)}  \frac{\hat{a}^*\left({n}/{k}  \right)}  {\hat{a}\left({n}/{k}  \right)}
  \left\{   ({d_n^{\hat\gamma^*}-1})/{\hat\gamma^*}
   -   ({d_n^{\hat\gamma}-1})/ {\hat\gamma} \right\}
  := A1+A2+A3,\label{eq-6.98A}
\end{eqnarray}
say.
The technical details when $\gamma_- =0$ and when $\gamma_- < 0$ are different, as mentioned earlier. For example, $(x^{\gamma} - 1)/\gamma$ needs to be interpreted as $\log x$ at  $\gamma=0.$
In this proof, we provide the details for the case $\gamma_- =0.$

It follows from
Remark 4.3.3 and (4.3.5) in \cite{Ha:Ferr:2006}, that
\begin{equation}
q_{\gamma}(t) \sim
\left\{ \begin{array}{ll}
\gamma^{-1} t^{\gamma}\log t & \\
       (\log t)^2/2, & \\
              \gamma^{-2}, & \\
\end{array}
\right.
\quad
\frac{t^{\gamma}-1}{\gamma} \sim
\left\{ \begin{array}{ll}
\gamma^{-1} t^{\gamma} & \\
       2 \log t, & \\
              -\gamma^{-1}, & \\
\end{array}
\right.
\quad
\frac{t^{\gamma}-1}{\gamma q_{\gamma}(t) } \sim
\left\{ \begin{array}{ll}
(\log t)^{-1}  & \gamma>0 \\
      2 (\log t)^{-1} , & \gamma=0 \\
              -\gamma, & \gamma<0 \\
\end{array}
\right.
\label{eq:mm_l12.1}
\end{equation}
as $t \rightarrow \infty$,
Then,
\[
\frac{d_n^{\hat\gamma}-1}{\hat\gamma q_{\hat\gamma}(d_n)} \sim  -\hat \gamma_- = -\gamma_- +  o_{P'}(1), \quad
\frac{1}{ q_{\hat\gamma}(d_n)} \sim (\hat \gamma_-)^2 =   (\gamma_-)^2+  o_{P'}(1).
\]
%
Substituting these into  A1 and A2 in (\ref{eq-6.98A}), we obtain
$A1=  (\gamma_-)^2 B_{k,n}'(1)  +  o_{P'}(1)$
by  Lemma \ref{l:r_3a},
and
$ A2 =  - (\gamma_-)A_n' +  o_{P'}(1) $
by  Lemma \ref{l:mm11}.

To evaluate $A3$, first notice that
$
({d_n^{\gamma}-1}){\gamma}^{-1} = \int_1^{d_n} s^{\gamma-1}ds.
$
Since
$k^{1/2}{\hat{a}^*(n/k)}/{\hat{a}(n/k)} = 1 + O_{P'}(1)$
by Lemma \ref{l:mm11}, we have
\begin{eqnarray*}
 A3 &=& \frac{k^{1/2}}{q_{\hat\gamma}(d_n)} \frac{\hat{a}^*\left(\frac{n}{k}  \right)}  {\hat{a}\left(\frac{n}{k}  \right)}        \left(   \frac{d_n^{\hat\gamma^*}-1}{\hat\gamma^*}   -   \frac{d_n^{\hat\gamma}-1}{\hat\gamma} \right)
 =    \frac{k^{1/2}}{q_{\hat\gamma}(d_n)} \int_1^{d_n} s^{\hat\gamma-1}(s^{\hat\gamma^*-\hat\gamma} - 1) d s \
+ o_{P'}(1).
\end{eqnarray*}
Let us ignore the $o_{P'}(1)$ for now and consider the other term.
By the mean value theorem,
$\log s^{\hat\gamma^*-\hat\gamma} - \log 1
 = \nu(s)^{-1}\{ s^{\hat\gamma^*-\hat\gamma} -1 \}$,
where 
$\nu(s)=s_1^{\hat\gamma^*-\hat\gamma} $ for some $s_1 \in[1,s]$.
Therefore,
\begin{eqnarray*}
 && {k^{1/2}}\{q_{\hat\gamma}(d_n)\}^{-1}\int_1^{d_n} s^{\hat\gamma-1}(s^{\hat\gamma^*-\hat\gamma} - 1) d s
 = k^{1/2}( \hat\gamma^*-\hat\gamma) + B_2
 %
\end{eqnarray*}
where $B_2 =  {k^{1/2}}\{q_{\hat\gamma}(d_n)\}^{-1}\int_1^{d_n} s^{\hat\gamma-1} \log s^{\hat\gamma^*-\hat\gamma}   (s_1^{\hat\gamma^*-\hat\gamma} -1)ds.$

Since $k^{-1/2}\log np_n =o(1)$ by assumption, we have
\begin{equation}
|\log d_n^{\hat\gamma^*-\hat\gamma}| =
 |(\hat\gamma^*-\hat\gamma) \log d_n|
 =(\hat\gamma^*-\hat\gamma)[\log k - \log np_n] = o_{P'}(1).
\label{eq-4.78c}
\end{equation}
Choose $n_0$ large enough such that $d_n >1$ for $n \geq n_0$.
Let $n > n_0$.
Then
\begin{eqnarray*}
|B_2|
& \leq &  |d_n^{\hat\gamma^*-\hat\gamma} -1| \frac{k^{1/2}}{q_{\hat\gamma}(d_n)}
 \int_1^{d_n} | s^{\hat\gamma-1} \log s^{\hat\gamma^*-\hat\gamma} |  ds
 =  |d_n^{\hat\gamma^*-\hat\gamma} -1| O_{P'}(1)\\
|\log d_n^{\hat\gamma^*-\hat\gamma}| &=&
 |(\hat\gamma^*-\hat\gamma) \log d_n|
 =(\hat\gamma^*-\hat\gamma)[\log k - \log np_n] = o_{P'}(1)
\end{eqnarray*}
since $k^{-1/2}\log np_n =o(1)$ by assumption.
Therefore, $B_2 = o_{P'}(1)$.
Then,
  $A_3 = k^{1/2}( \hat\gamma^*-\hat\gamma)+ o_{P'}(1) = \Gamma_n + o_{P'}(1)$ by (\ref{eq:mm l9.3}).

Since $q_{\hat{\gamma}^*}(d_n)/q_{\hat{\gamma}}(d_n) = 1 + o_{P'}(1)$
(see, \citealt{Ha:Ferr:2006}, Corollary 4.3.2, page135), and
$\hat{a}^*\left({n}/{k}\right)\{\hat{a}\left({n}/{k}\right)\}^{-1} = 1 + o_{P'}(1)$ by Lemma \ref{l:mm11},
using  (\ref{eq-6.98A}),  we have
\begin{eqnarray*}
 k^{1/2}  \frac{\hat{x}_n^*-\hat{x}_n}{q_{\hat\gamma}(d_n)\hat{a}\left({n}/{k}  \right)} & = &  A1+A2+A3 =
\Gamma_n' + (\gamma_-)^2 B_{k,n}'(1) - (\gamma_-)A_n' +  o_{P'}(1).
\end{eqnarray*}
\end{proof}

\subsection{Estimation of tail probability}
In this subsection, we prove the main results for constructing confidence intervals for an unknown
 tail probability.
Let $x_n$ be a given  sequence of real numbers.
The results in this subsection are for constructing a bootstrap confidence interval
for the unknown tail probability
$p_n := pr(X \geq x_n).$
In the proof of Lemma \ref{l:mm12}, we used the two quantities
$$\hat{x}_{p_n}  = \hat{b}(n/k) + \hat{a}(n/k) \{d_n^{\hat{\gamma}} - 1\}/\hat{\gamma}, \quad
\hat{x}_{p_n}^*  = \hat{b}^*(n/k) + \hat{a}^*(n/k) \{d_n^{\hat{\gamma}^*} - 1\}/\hat{\gamma}^*.$$
In the technical arguments provided below, we use these two quantities and
 results about their properties obtained in the proof of Lemma \ref{l:mm12}.
One difference worthy of noting is that,
 in the context of this subsection, $\hat{x}_{p_n}$ is an estimator of the known quantity $x_n$.
Therefore,  while we use $\hat{x}_{p_n}$ in the following technical details, it is used only to facilitate
 using the previously obtained results in the current proofs, not as an empirically useful statistic.

Recall that 
\begin{equation}
\hat{p}_{n}  = \frac{k}{n} \left(  \max \left\{  0, 1+\hat{\gamma} \frac{ x_{n}-\hat{b}\left(\frac{n}{k}\right)}{  \hat{a}\left(\frac{n}{k}  \right) } \right\} \right)^{-1/\hat{\gamma}}.
\label{eq-4.79a}
\end{equation}
and $w_{\gamma}(t) = t^{-\gamma}\int_1^t s^{\gamma-1}\log s \ ds$, $t>0$.

\begin{lemma}
\label{l:mm13}
Suppose that the conditions of Theorem \ref{th:mm_p_boot} are satisfied.
Then,
\begin{eqnarray}
\frac{k^{1/2}}{w_{\hat\gamma}(\hat{d}_n)} \left( \frac{\hat{p}_n^*}{\hat p_n}-1 \right)
 &=& \Gamma_n' + (\gamma_-)^2B_{k,n}'(1) - (\gamma_-)A_n' +  o_{P'}(1).
\label{eq:r_p_wcc} \\
\frac{k^{1/2}}{w_{\hat{\gamma}^*}(\hat{d}_n^*)} \left( \frac{\hat{p}_n^*}{\hat p_n}-1 \right)
 &=& \Gamma_n' + (\gamma_-)^2B_{k,n}'(1) - (\gamma_-)A_n' +  o_{P'}(1).
\label{eq:r_p_wcc-2}
\end{eqnarray}
\end{lemma}

\begin{proof}
Let
$$ Y_n = \left( \frac{\hat{\gamma}} {d_n^{\hat\gamma}} \frac{ x_{n}-\hat{x}_{n}}{\hat{a}\left(\frac{n}{k}  \right) }   \right),
\qquad
Y_n^* =  \left(\frac{\hat{\gamma}^*} {d_n^{\hat\gamma^*}} \frac{ x_{n}-\hat{x}_{n}^*}{\hat{a}^*\left(\frac{n}{k}  \right)  }  \right).$$
Using the definitions of $\hat{p}_n, \hat{x}_{n}, \hat{p}_n^*,$ and $\hat{x}_n^*$, and simplifying the expressions, we obtain
\begin{eqnarray}
\hat{p}_{n}
= \frac{k}{nd_n} \left( 1  + Y_n    \right)^{-1/\hat{\gamma}},
& \hat{p}_{n}^* = \frac{k}{nd_n}\left( 1  + Y_n^*  \right)^{-1/\hat{\gamma}^*}, &
\frac{\hat{p}_{n}^*}{\hat{p}_{n}} = \left( 1  +Y_n^*  \right)^{-1/\hat{\gamma}^*}
   \left( 1  + Y_n  \right)^{1/\hat{\gamma}}.
\label{eq-4.82a}
\end{eqnarray}
By the mean value theorem,
$
\log  \left(1  +Y_n\right) - \log 1 = \zeta_n^{-1}Y_n,
$
where $\zeta_n$ lies between $1$ and $1+Y_n$.
Using the previously established results and assumptions about the rates of convergence of $d_n$,
and (\ref{eq:mm_l12.1}), it may be verified that $Y_n= o_{P'}(1)$.
Therefore, $\zeta_n = 1 + o_{P'}(1)$ and
$
\log  \left(1  +Y_n\right) = (1 + o_{P'}(1))Y_n.
$

Similarly, by the mean value theorem,
$
\log \left(1  + Y_n^* \right) - \log 1 = \xi_n^{-1}Y_n,
$
where
$\xi_n$ lies between $1$ and $Y_n^*$.
As in the previous proof for $Y_n = o_{P'}(1)$, it may be verified that $Y_n^* = o_{P'}(1),$
 $\xi_n = 1 + o_{P'}(1),$
and
$ \log \left(1  +Y_n^*\right) = (1 + o_{P'}(1))Y_n^*$.

Now,
\begin{eqnarray*}
&&\frac{k^{1/2}}{w_{\hat\gamma}(d_n)}   \log \frac{\hat{p}_{n}^*}{\hat{p}_{n}}
 = \frac{k^{1/2}}{w_{\hat\gamma}(d_n)}  \Big[ - \{\hat{\gamma}^*\}^{-1} \log \left(1  + Y_n^*  \right)
       +  \{\hat{\gamma}\}^{-1} \log  \left(1  + Y_n \right)\Big]  \\
&=&  k^{1/2}  \frac{ \hat{x}_{n}^* -\hat{x}_{n}}{q_{\hat\gamma}(d_n)\hat{a}\left(\frac{n}{k}  \right) } \frac{\hat{a}\left(\frac{n}{k}  \right) }{\hat{a}^*\left(\frac{n}{k}  \right) }d_n^{\hat\gamma-\hat\gamma^*} (1 + o_{P'}(1)) +
 k^{1/2}\frac{ x_{n}-\hat{x}_{n}}{q_{\hat\gamma}(d_n) \hat{a}\left(\frac{n}{k}  \right) }  \left( 1- \frac{\hat{a}\left(\frac{n}{k}  \right) }{\hat{a}^*\left(\frac{n}{k}  \right) }d_n^{\hat\gamma-\hat\gamma^*}  \right)(1 + o_{P'}(1))  
\end{eqnarray*}
By (\ref{eq-4.78c})
$| d_n^{\hat\gamma^*-\hat\gamma} -1|=o_{P'}(1)$.
By (\ref{eq:mm l11}), $\hat{a}^*(n/k) /\hat a (n/k) =1 + O_P'(k^{-1/2})= 1+o_{P'}(1).$ By Theorem \ref{th:mm_x_boot},
$
 k^{1/2}q_{\hat\gamma}^{-1}(d_n) \hat{a}^{-1}(n/k) ( x_{n}-\hat{x}_{n})   = O_{P'}(1).
$
By Lemma \ref{l:mm13},
\[
 k^{1/2}  \frac{ \hat{x}_{n}^* -\hat{x}_{n}}{q_{\hat\gamma}(d_n)\hat{a}\left(\frac{n}{k}  \right) }  = \Gamma_n' + (\gamma_-)^2B_{k,n}'(1) - (\gamma_-)A_n' +  o_{P'}(1).
\]
Therefore,
\begin{eqnarray}
\frac{k^{1/2}}{w_{\hat\gamma}(d_n)}   \log \frac{\hat{p}_{n}^*}{\hat{p}_{n}}
 &= &\left( \Gamma_n' + (\gamma_-)^2B_{k,n}'(1) - (\gamma_-)A_n' +  o_{P'}(1)\right) (1+o_{P'}(1)) + o_{P'}(1) \notag \\
&=&\Gamma_n' + (\gamma_-)^2B_{k,n}'(1) - (\gamma_-)A_n' +  o_{P'}(1).
\label{eq:mm_l13.3}
\end{eqnarray}

It follows from (\ref{eq:mm_l13.3}) that $\log (\hat{p}_{n}^*/ \hat{p}_{n}) =  o_{P'}(1)$, and hence $|\hat{p}_{n}^*/ \hat{p}_{n} -1 |= o_{P'}(1)$.
Therefore, by the delta method, we have (\ref{eq:r_p_wcc}).
Let $\hat{d}_n = k/(n \hat{p}_n)$  and $\hat{d}_n^* = k/(n \hat{p}_n^*)$.
Then  (see Corollary 4.4.4, \citealt{Ha:Ferr:2006})
since $\hat{\gamma}^* - \hat{\gamma} = O_{P'}(k^{-1/2})$ and $\hat{p}_{n}^*/ \hat{p}_{n} = 1+ o_{P'}(1)$, we have
$\{w_{\hat{\gamma}}(\hat{d}_n)/w_{{\gamma}}({d}_n)\}$ and
$\{w_{\hat{\gamma}^*}(\hat{d}_n^*)/w_{\hat{\gamma}}(\hat{d}_n)\}$ converge to $1$, in probability.
Therefore, we also have
\begin{eqnarray}
\frac{k^{1/2}}{w_{\hat{\gamma}^*}(\hat{d}_n^*)}   \log \frac{\hat{p}_{n}^*}{\hat{p}_{n}}
&=&\Gamma_n' + (\gamma_-)^2B_{k,n}'(1) - (\gamma_-)A_n' +  o_{P'}(1).
\label{eq:mm_l13.3b}
\end{eqnarray}
\end{proof}


\begin{lemma}
\label{l:mm14}
Let $\bfV=(V_1,V_2, \ldots)$  and $\bfV'=(V_1',V_2', \ldots)$  be two sequences of random variables as
in (\ref{eq:bqp}),
 $S_n(\bfV,\bfV')$ be a function of $V_1, \ldots, V_n$
 and $V_1', \ldots, V_n'$, and $S_n(\bfV,\bfV')=o_{P'}(1)$.
 Then $S_n(\bfV,\bfV')=o_{P^*}(1)$, in probability, where
%
 $P^*$ denotes the probability conditional on $V'$.
\end{lemma}

\begin{proof}
Let $\epsilon >0$ and  $\delta>0$ be given. Suppose that $S_n(\bfV,\bfV')=o_{P'}(1) $.
Let $Z_n(\bfV; \delta)=P^*\left[ S_n(\bfV,\bfV') \geq \delta | \bfV  \right]$,
where $P^*$ denotes the conditional probability given $\bfV$.
Then
$E[Z_n^2(\bfV; \delta)] \leq E[Z_n(\bfV; \delta)]
= P'(S_n(\bfV,\bfV' \geq \delta)
\stackrel{P'}{\rightarrow} 0$ as $n \rightarrow \infty.$
and, by Chebyshev inequality, $P_{\bfV}[ Z_n(\bfV; \delta) >\epsilon] =o_{P'}(1)$.
Therefore,
$P_{\bfV} \left( P_{\bfV'|\bfV}  \left[|S_n(\bfV,\bfV')| \geq \delta | \bfV\right] > \epsilon \right) \rightarrow 0 \ \mbox{as} \ n \rightarrow \infty,$
and hence $S_n(\bfV,\bfV')=o_{P^*}(1)$ in probability $[P]$.

\end{proof}


\begin{proposition}
\label{p:mm1}
$$(P_n', Q_n', A_n', \Gamma_n') \stackrel{d}{\rightarrow} (P, Q, A, \Gamma), \quad
(\tilde{P}_n, \tilde{Q}_n, \tilde{A}_n, \tilde{\Gamma}_n) \stackrel{d}{\rightarrow} (P, Q, A, \Gamma), \ \mbox{as } n \rightarrow \infty.$$
\end{proposition}
\begin{proof}
 Since$
B_{k,n}'(s)
 =  W(s)- ({ks}/{n})W\left({n}/{k}\right), \ (0\leq s \leq 1),
$
in distribution,
  we have
\begin{eqnarray}
\frac{B_{k,n}'(s)}{s^{\gamma_-+1}}  &\stackrel{d}{=}& \frac{W(s)}{s^{\gamma_-+1}}- \frac{k}{n}W\left(\frac{n}{k}\right)s^{-\gamma_-} 
= \frac{W(s)}{s^{\gamma_-+1}} -  \left(\frac{k}{n} \right)^{1/2} \left(\frac{k}{n} \right)^{1/2} W\left(\frac{n}{k}\right)s^{-\gamma_-}. \label{eq:mm_01}
\end{eqnarray}
Since $(k/n)^{1/2}W(n/k)$ is bounded in probability,
and $s^{-\gamma_-}$ is uniformly continuous on $[0,1]$, we have
$\sup_{0 \leq s \leq 1} |\left( {k}/{n} \right)^{1/2} \left({k}/{n} \right)^{1/2} W\left({n}/{k}\right)s^{-\gamma_-} = $ $ o_{P'}(1).$
Therefore, the right hand side of  (\ref{eq:mm_01}) is tight and hence it converges weakly to $W(s)s^{-(\gamma_-+1)}$,
($0 < s <1$).
\end{proof}



\subsection{Proofs of the theorems}
\label{sec-3.4}

In this subsection, we provide the proofs of the theorems in the main paper.
These results are in terms of the conditional probability $P^*$ given the observations
$X_1, \dots, X_n.$

\smallskip

\noindent \textit{Prof of Theorem \ref{th:mm_b_boot}}.
For a statement and proof of the corresponding non-bootstrap version of this theorem,
see Lemma 4.3 in \cite{deHa:Resn:1993}; according to this Lemma, under the conditions of this theorem,
$k^{1/2}\{{\hat b(n/k)-b(n/k)}\}/{ a(n/k)}$ converges in distribution to $B$
as $n \rightarrow \infty$.
By Lemma \ref{l:r_3a},
$
[k^{1/2}/\hat a(n/k)]\{\hat b^*(n/k)-\hat b(n/k)\} =  B_{k,n}'(1)+ o_{P'}(1).
$
Then, by Lemma \ref{l:mm14},
$
[k^{1/2}/\hat a(n/k)] \{\hat b^*(n/k)-\hat b(n/k)\}= B_{k,n}'(1)+ o_{P^*}(1),
$
 in probability $[P]$.
Now,  the result follows by Proposition \ref{p:mm1}.
Since $B_n$ is independent of $\bfV$,  and hence of ${\bfX}$,
 it follows from Proposition \ref{p:mm1}
  that
$ [k^{1/2}/\hat a(n/k)]\{\hat b^*(n/k)-\hat b(n/k)\}  \stackrel{d^*}{\rightarrow} B, $
in probability, as $n \rightarrow \infty.$
\hfill $\blacksquare$


\smallskip

\noindent \textit{Prof of Theorem \ref{th:mm_gamma_boot}}.
For a statement and proof of the corresponding non-bootstrap version of this theorem,
see Lemma 4.6 in \cite{deHa:Resn:1993}; according to this Lemma, under the conditions of this theorem,
$k^{1/2}\{\hat{\gamma}-\gamma\} \stackrel{d}{\rightarrow} \Gamma$
as $n \rightarrow \infty$.
By (\ref{eq:mm l9.3}),
$k^{1/2}(\hat{\gamma}^* - \hat{\gamma}) =  \Gamma_n' +o_{P'}(1).$
Then, by Lemma \ref{l:mm14},
$
k^{1/2}(\hat{\gamma}^* - \hat{\gamma}) =  \Gamma_n'+ o_{P^*}(1),
$
 in probability $[P]$;
note that the remainder term is $o_{P^*}(1)$, not $o_{P'}(1).$
Since $\Gamma_n'$ is independent of ${\bfV}$,  and hence of ${\bfX}$,
the result follows by Proposition \ref{p:mm1}.
\hfill $\blacksquare$

\smallskip

\noindent \textit{Proof of Theorem \ref{th:mm_a_boot}}.
For this proof, $p_n$ is given and $x_n$ is the unknown quantile corresponding to $p_n$.
For a statement and proof of the corresponding non-bootstrap version of this theorem,
see Lemma 4.7 in \cite{deHa:Resn:1993}; according to this Lemma, under the conditions of this theorem,
$k^{1/2}\left(\{{\hat{a}(n/k)}/{{a}(n/k)}\} - 1\right) \stackrel{d}{\rightarrow} A$
as $n \rightarrow \infty$.
By Lemma \ref{l:mm11},
$
k^{1/2}\left(\{{\hat{a}^*(n/k)}/{\hat{a}(n/k)}\} - 1\right) =A_n' +o_{P'}(1).
$
Then, by Lemma \ref{l:mm14},
$
k^{1/2}\left( \{{\hat{a}^*(n/k)}/{\hat{a}(n/k)}\} - 1\right) =A_n' + o_{P^*}(1)
$
 in probability $[P]$.  Now, the result follows by
 Proposition \ref{p:mm1}.
 \hfill $\blacksquare$

\smallskip

\noindent \textit{Proof of Theorem \ref{th:mm_x_boot}}.
For this proof, $x_n$ is known, and $p_n= P(X> p_n)$ is the unknown tail probability.
For a statement and proof of the corresponding non-bootstrap version of this theorem,
see Theorem 4.3.1 in \cite{Ha:Ferr:2006}; according to this result,
\begin{equation}
\frac{k^{1/2}}{q_{\hat\gamma}(d_n)}  \frac{\hat{x}_{n}-x_{n}}{\hat a\left(\frac{n}{k}  \right)}  \stackrel{d}{\rightarrow}  \Gamma + (\gamma_-)^2B - (\gamma_-)A \quad \mbox{ as } n \rightarrow \infty.
\end{equation}
under the conditions of Theorem \ref{th:mm_x_boot}.
By Lemma \ref{l:mm12},
\[
[{k^{1/2}}/{q_{\hat\gamma}(d_n)}]  [\{\hat{x}_n^*-\hat{x}_n\}/{\hat{a}\left({n}/{k} \right)} ] = \Gamma_n' + (\gamma_-)^2B_{k,n}'(1) - (\gamma_-)A_n' +  o_{P'}(1).
\]
The Lemma also shows that the result holds with the scaling factor,
${q_{\hat\gamma}(d_n)}\hat{a}\left({n}/{k}  \right)$ replaced
by its bootstrap counterpart
${q_{\hat{\gamma^*}}(d_n)}\hat{a}^*\left({n}/{k}  \right)$.
Then, by Lemma \ref{l:mm14},
\[
\{{k^{1/2}}/{q_{\hat{\gamma^*}}(d_n)}\}
 \{ (\hat{x}_n^*-\hat{x}_n)/{\hat{a}^*\left({n}/{k}  \right)} \}
    = \Gamma_n' + (\gamma_-)^2B_{k,n}'(1) - (\gamma_-)A_n' + o_{P^*}(1)
 \]
 in probability $[P]$. Therefore, by Proposition \ref{p:mm1},
$$
\{{k^{1/2}}/{q_{\hat{\gamma^*}}(d_n)} \} \{(\hat{x}_n^*-\hat{x}_n)/{\hat{a}^*\left({n}/{k}  \right)} \} \stackrel{d^*}{\rightarrow}  \Gamma + (\gamma_-)^2B - (\gamma_-)A.
\hfill  \blacksquare
$$

\smallskip

\noindent \textit{Proof of Theorem \ref{th:mm_p_boot}}.
For a statement and proof of the corresponding non-bootstrap version of this theorem,
 see Theorem 4.4.1 in \cite{Ha:Ferr:2006};
 it is shown there that
 \begin{equation}
\frac{k^{1/2}}{w_{\hat\gamma}(d_n)} \left( \frac{\hat{p}_n}{p_n}-1 \right) \stackrel{d}{\rightarrow} \Gamma + (\gamma_-)^2B - (\gamma_-)A.
\end{equation}

By Lemma \ref{l:mm13},
\[
\{ {k^{1/2}}/{w_{\hat\gamma}(\hat{d}_n)}\} \{ {\hat{p}_n^*}/{\hat p_n}-1 \}  = \Gamma_n' + (\gamma_-)^2B_{k,n}'(1) - (\gamma_-)A_n' +  o_{P'}(1).
\]
Then, by Lemma \ref{l:mm14},
\[
\{{k^{1/2}}/{w_{\hat\gamma}(\hat{d}_n)}\} \left( {\hat{p}_n^*}/{\hat p_n}-1 \right)  = \Gamma_n' + (\gamma_-)^2B_{k,n}'(1) - (\gamma_-)A_n' + o_{P^*}(1)
 \]
 in probability $[P]$. Therefore,
$
\{ {k^{1/2}}/{w_{\hat\gamma}(\hat{d}_n)} \} \left( {\hat{p}_n^*}/{\hat p_n}-1 \right)  \stackrel{d^*}{\rightarrow}  \Gamma + (\gamma_-)^2B - (\gamma_-)A,
$
in probability, by Proposition \ref{p:mm1}.
Again, the result holds with the scaling factor
${w_{\hat\gamma}(\hat{d}_n)}$
 replaced by its bootstrap counterpart
 ${w_{\hat{\gamma^*}}(\hat{d}_n^*)}$ where
 $d_n^* = k/(n \hat{p}_n^*).$
 \hfill $\blacksquare$

\bibliographystyle{natbib}
\bibliography{Regression}